\documentclass[a4paper]{article}
\usepackage{amsfonts}
\usepackage{amsbsy}
\usepackage[T1]{fontenc}
\usepackage[latin9]{inputenc}
\usepackage{geometry}
\usepackage{amsmath}
\usepackage{amsthm}
\usepackage{amssymb}
\usepackage{setspace}
\usepackage{esint}
\usepackage{geometry}
\usepackage{fancyhdr}
\usepackage{amsmath,amsthm,amssymb}
\usepackage{graphicx}
\usepackage{hyperref}
\usepackage{xargs}
\usepackage{xifthen}
\usepackage{amscd}
\usepackage{xspace}
\usepackage{verbatim}
\usepackage{amstext}
\usepackage{color}

\newcommand{\R}{\mathbb R}

\newcommand{\N}{\mathbb N}
\newcommandx{\inner}[2]{\left<{#1},{#2}\right>}
\newcommandx{\norm}[1][1=\cdot]{\left|{#1}\right|}
\newcommandx{\supnorm}[2][1=\cdot,2=]{\nnorm[#1]_{\infty,#2}}
\newcommandx{\nnorm}[1][1=\cdot]{\left|\left|{#1}\right|\right|}

\newcommandx{\Set}[2][2=]{
    \ifthenelse{\isempty{#2}}
        {\left\{ {#1} \right\}}
        {\left\{  {#1}  \, \middle| \, {#2} \right\}}
}

\DeclareMathOperator{\dist}{dist}

\DeclareMathOperator{\Supp}{Supp}
\DeclareMathOperator{\ess-sup}{ess-sup}
\DeclareMathOperator{\essinf}{ess-inf}
\DeclareMathOperator{\diver}{div}
\DeclareMathOperator{\image}{Im}
\DeclareMathOperator{\sign}{sign}

\numberwithin{equation}{section}
\numberwithin{figure}{section}
\theoremstyle{plain}

\usepackage{babel}
\newtheorem{thm}{Theorem}[section]
 \newtheorem{cor}[thm]{Corollary}
 \newtheorem{lem}[thm]{Lemma}
 \newtheorem{prop}[thm]{Proposition}
 
 \newtheorem{defn}[thm]{Definition}
 \newtheorem{rem}[thm]{Remark}
 \newtheorem{assume}{Assumption}

\begin{document}
\global\long\def\R{\mathbb{R}}
\title{$L^1$-contraction property of Entropy Solutions for Scalar Conservation
Laws with Minimal Regularity Assumptions on the Flux}

\author{Paz Hashash}
\date{\today}
\maketitle
\begin{abstract}
This paper is concerned with entropy solutions of scalar conservation laws of the form
\[
\partial_t u + \operatorname{div} f = 0 \quad \text{in } \mathbb{R}^d \times (0, \infty),
\]
where the flux \( f = f(x, u) \) depends explicitly on the spatial variable \( x \).  
Using an extension of Kruzkov's doubling variable method, we establish contraction properties of entropy solutions under minimal regularity assumptions on the flux, as well as the uniqueness of entropy solutions.  
The flux is assumed to be locally Lipschitz, along with some additional conditions.
\end{abstract}
\tableofcontents{}

\section{\textbf{Introduction}}
A scalar conservation law is a quasilinear partial differential equation of the form  
\[
\partial_{t} u + \diver f = 0
\]  
in the half-space \((x, t) \in \mathbb{R}^d \times (0, \infty)\). Here, \(u = u(x, t) \in \mathbb{R}\) is the unknown function, \(f = f(x, u)\) is the flux function, and the divergence operator \(\diver\) is taken with respect to the spatial variable \(x\).  

It is well known that, in general, global (in time) differentiable solutions for the scalar conservation law do not exist. Consequently, the concept of weak solutions (or integral solutions) is introduced. Weak solutions are obtained by multiplying the scalar conservation law by a test function and formally\footnote{This is because \(u\) is generally not differentiable.} integrating by parts to transfer the derivatives from \(u\) and \(f\) to the test function.  

As demonstrated in \cite{Kruzkov}, for sufficiently smooth flux \(f\), the existence of a weak solution can be established by first proving the existence of solutions to the viscous parabolic equation  
\[
\partial_{t} u_{\epsilon} + \diver f(x, u_{\epsilon}) = \epsilon \Delta u_{\epsilon}, \quad \epsilon \in (0, \infty),
\]  
and then taking the limit \(\epsilon \to 0^+\). For regular flux \(f\) and initial data \(u_0\), the viscous solutions \(u_{\epsilon}\) are unique and regular; however, the limiting weak solution is generally not regular.  

Weak solutions to scalar conservation laws are not necessarily unique. To address this, the concept of "entropy solutions" is introduced. Entropy solutions are weak solutions that satisfy specific "entropy conditions." The \(L^1\)-contraction property of entropy solutions is crucial in proving their uniqueness within the class of weak solutions. For a general introduction to entropy solutions for scalar conservation laws, refer to \cite{dafermos, Evans,GagneuxMadauneTort1996,GodlewskiRaviart}.  

The \(L^1\)-contraction property was first mentioned, possibly, by Oleinik in 1957 (translated into English in 1963, see \cite{Oleinik}). This property applies to solutions admissible under Oleinik's one-sided Lipschitz condition for the one-dimensional case with convex flux \(f\) in \(u\), but inhomogeneous, meaning that \(f\) explicitly depends on the spatial variable \(x\). Vol'pert later provided the first multi-dimensional result, though limited to BV solutions \cite{vol'pert}, using pointwise jump admissibility conditions from \cite{Gelfand}, rather than an entropy formulation.  

Kruzhkov introduced a method for proving \(L^1\)-contraction via entropy inequalities \cite{Kruzkov}. In this approach, the entropy inequality can be viewed as a localized contraction inequality (the so-called Kato inequality) applied to the candidate solution and the constant solution \(u(x,t) \equiv k\). For inhomogeneous flux, the constant \(k\) satisfies the conservation law with a source term \(\text{div}_x f(x, k)\), which introduces an additional term. Obvious reference solutions can replace constants in the analysis \cite{Audusse, Colombo}. This perspective is crucial in the study of discontinuous fluxes via adapted entropy inequalities \cite{Audusse, andreianov2011}.  

The semigroup method developed in the early 1970s inherently leads to \(L^1\)-contraction \cite{Crandall, Benilan1}. The kinetic formulation, developed in the 1990s, offers a genuinely alternative approach (compared to the doubling of variables) for proving \(L^1\)-contraction for entropy solutions with pure \(L^1\) data \cite{Perthame1, Perthame2}. The inhomogeneous case was later developed in \cite{Dalibard}.  

Various cases of fluxes with explicit dependence on the spatial variable have been studied in the literature. Otto \cite{otto} considered fluxes of the form \(f(x, u) = g(u)v(x)\), where \(v(x)\) is a given velocity field. Seguin and Vovelle \cite{SeguinAnslysis} established \(L^1\)-contraction for fluxes of the form \(f(x, u) = k(x)u(1 - u)\), where \(k(x)\) is discontinuous. More recently, Bachmann and Vovelle \cite{BachmannJulien} proved \(L^1\)-contraction for fluxes of the form \(f(x, u) = g(x, u) + h(u)\), where \(g(x, u)\) is discontinuous, with \(g(x, u) = g_L(u)\) for \(x < 0\) and \(g(x, u) = g_R(u)\) for \(x > 0\), and \(g_L \neq g_R\) are Lipschitz functions. \(h(u)\) is also Lipschitz. Many other works have addressed cases with discontinuous fluxes; see, e.g., \cite{andreianov2015,bulicek,crasta, diehl, jimenez, mitrovic, panov, shen2012, shen2013}.  
In \cite{BLe}, Ben-Artzi and LeFloch established the \(L^1\)-contraction property on manifolds, assuming the flux is smooth. In \cite{Lengeler-Muller}, Lengeler and M{\"u}ller proved \(L^1\)-contraction assuming that \(f = f(x, u)\), as well as its derivative \(\partial_u f\), are continuously differentiable.  

In \cite{Kruzkov}, Kruzkov used the doubling of variables method to establish the \(L^1\)-contraction property via entropy inequalities. He proved the property under the assumption that the flux \(f = f(x, k)\) is continuously differentiable, with spatial derivatives \(\partial_{x_i} f(x, k)\), \(1 \leq i \leq d\), $i\in\N$, that are Lipschitz continuous in \(k\) (see details below).  
In this paper, we establish the \(L^1\)-contraction property for scalar conservation laws of the form  
\[
\partial_{t} u + \diver f(x, u) = 0,
\]  
where the flux \(f = f(x, k)\) explicitly depends on the spatial variable \(x\) and satisfies minimal regularity conditions. Our proof uses a suitable extension of Kruzkov's method.
We establish the \(L^1\)-contraction property under the assumption that the flux \(f = f(x, k)\) satisfies certain regularity hypotheses:

\begin{assume}(Assumptions on the flux $f$) 
\label{ass:assumption on the flux} For a function $f:\R^d\times\R\to \R^d, f=f(x,k),$ we assume:
\\
1. $f$ is locally Lipschitz on $\R^d\times\R$.
\\ 
2a. There exists a set $\Theta\subset\R^d$ such that $\mathcal{L}^d(\Theta)=0$, where $\mathcal{L}^d$ is the $d$-dimensional Lebesgue measure, and for every $k\in \R$, the function $f(\cdot,k):\R^d\to \R^d$ is differentiable at every point $x\in \R^d\setminus\Theta$. Moreover, we assume that for every $x\in \R^d\setminus\Theta$, the function $k\mapsto D_xf(x,k)$ is continuous, where $D_xf(x,k)$ is the differential of $f(\cdot,k)$ at $x$.
\\
2b. Moreover, at every such point $x$, we assume that for every compact set $K\subset\R$ we have
\begin{equation}
\label{eq:uniform convergence of differential quotient}
\lim_{y\to x}\left(\sup_{k\in K}\frac{|f(y,k)-f(x,k)-D_xf(x,k)(y-x)|}{|y-x|}\right)=0.
\end{equation} 
\end{assume}

In this article, our objective is to establish the $L^1$-contraction property while imposing minimal regularity assumptions on the flux function $f$. We would like to highlight two key distinctions between Kruzkov's assumptions on the flux \cite{Kruzkov} and Assumption 1 above:
\begin{enumerate}
\item In Kruzkov's proof, the flux $f= f(x,k)$ is assumed to be continuously differentiable, while we assume it to be locally Lipschitz. Continuous differentiability implies \eqref{eq:uniform convergence of differential quotient} (see Remark \ref{rem:uniform continuity of the differential quotient}).

\item In Kruzkov's proof, the derivatives of the flux $f = f(x,k)$ with respect to the spatial variables $x_i, 1\leq i\leq d$, $i\in\N$, are assumed to be Lipschitz continuous with respect to $k$; while, we only assume that the derivatives of the flux with respect to the spatial variable exist almost everywhere\footnote{The phrases "almost everywhere" and "almost every" always refers to Lebesgue measure of the relevant dimension. For example, if we write "for almost every $x\in\R^d$", we mean almost everywhere with respect to the $d$-dimensional Lebesgue measure. Similarly, if we write "for almost every $(x,t)\in\R^d\times (0,\infty)$", we mean almost everywhere with respect to the $(d+1)$-dimensional Lebesgue measure, and so on.} and they are continuous functions of the variable $k$.    
\end{enumerate}

The article is organized as follows: Section \ref{sec:Remarks about the assumptions of the flux $f$} is devoted to a discussion about Assumption \ref{ass:assumption on the flux}. Section \ref{sec:Lipschitz analysis} focuses on Lipschitz analysis. In Section \ref{sec:scalar conservation laws}, we present the main concepts related to scalar conservation laws and entropy solutions. In Section \ref{sec:L1-contraction property}, we prove the $L^1$-contraction property for entropy solutions.
\\
\\
{\it To the best of the author's knowledge, the primary contribution of this paper is the proof of the localized contraction property (Lemma \ref{lem:Kato's inequality}) under Assumption \ref{ass:assumption on the flux} regarding the flux. All other content is included for the sake of completeness and self-containment.}
\section{Remarks about the assumptions on the flux $f$}
\label{sec:Remarks about the assumptions of the flux $f$}
In this section we give some remarks about the mentioned above Assumption \ref{ass:assumption on the flux}.

\begin{rem}(Uniform differentiability almost everywhere)
\label{rem:Uniform differentiability almost everywhere}
Assume $f:\R^d\times\R\to \R^d, f=f(x,k),$ is a measurable function. Let us distinguish between the following two assertions:
\begin{enumerate}
\item For every $k\in \R$, there exists $\Theta_k\subset\R^d$ such that $\mathcal{L}^d(\Theta_k)=0$, and for every $k\in \R$, the function $x\mapsto f(x,k)$ is differentiable at every $x\in\R^d\setminus \Theta_k$;

\item \textbf{Uniform differentiability almost everywhere}: there exists $\Theta\subset\R^d$ such that $\mathcal{L}^d(\Theta)=0$, and for every $k\in \R$, the function $x\mapsto f(x,k)$ is differentiable at every $x\in\R^d\setminus \Theta$.
\end{enumerate}
The second assertion implies the first assertion, but not the opposite. In case $f$ is a Lipschitz function, Rademacher's Theorem (see Theorem \ref{thm:Rademacher's theorem}) tells us that for every $k\in \R$, the function $f(\cdot,k):\R^d\to \R^d$ is differentiable almost everywhere, and we get the first assertion. However, it does not imply uniform differentiability almost everywhere. Notice that item 2a in Assumption \ref{ass:assumption on the flux} requires that the flux $f$ has the property of uniform differentiability almost everywhere.
\end{rem}

\begin{rem}(Continuous differentiability of the flux $f$ implies property \eqref{eq:uniform convergence of differential quotient})
\label{rem:uniform continuity of the differential quotient}
In this remark, we want to show that if the flux $f$ as in Assumption \ref{ass:assumption on the flux} is continuously differentiable, then the property \eqref{eq:uniform convergence of differential quotient} holds.

It is enough to prove it for scalar functions $f$. Let $f:\R^d\times\R\to \R$ be a continuously differentiable function. Then, for $x\in \R^d$ and a compact set $K\subset\R$, $D_xf$ is uniformly continuous on $\overline{B}_1(x)\times K$, so for arbitrary $\xi\in (0,\infty)$ there exists $\delta\in(0,1)$ such that, for every  $(x_1,k_1),(x_2,k_2)\in \overline{B}_1(x)\times K$, such that $|(x_1,k_1)-(x_2,k_2)|<\delta$ we have $|D_xf(x_1,k_1)-D_xf(x_2,k_2)|\leq \xi$. Therefore, for every $(y,k)\in B_\delta(x)\times K$, $y\neq x$, we get by the Fundamental Theorem of Calculus
\begin{multline}
\frac{|f(y,k)-f(x,k)-D_xf(x,k)(y-x)|}{|y-x|}=\frac{\left|\int_0^1\frac{d}{dt}f(ty+(1-t)x,k)dt-D_xf(x,k)(y-x)\right|}{|y-x|}
\\
\leq \int_0^1\big|D_xf(ty+(1-t)x,k)-D_xf(x,k)\big|dt\leq \xi.
\end{multline}
Therefore,
\begin{equation}
\lim_{y\to x}\left(\sup_{k\in K}\frac{|f(y,k)-f(x,k)-D_xf(x,k)(y-x)|}{|y-x|}\right)=0.
\end{equation}     
\end{rem}

\begin{rem}[Comparison with alternative setting of conditions]
In \cite{KarlsenRisebro2000}, Karlsen and Risebro established the $L^1$-contraction property under the following conditions on the flux $f:\R^d\times\R\to \R^d$:
\begin{equation}
\label{eq:set of conditions 1}
\begin{cases}
1)\,\, f(\cdot,u)\in W^{1,1}_{\text{loc}}(\R^d)
\\
2)\,\,\diver_xf(\cdot,u)\in L^\infty(\R^d)
\\
3)\,\,|f(x,u)-f(x,v)|\leq C|u-v|
\\
4)\,\,|\diver_xf(x,u)-\diver_xf(x,v)|\leq C|u-v|
\end{cases} \forall u,v\in \R,\, x\in \R^d,
\end{equation}
where the constant $C$ does not depend on $x,u,v$. They require an additional condition: for all $x,y\in\R^d$ and $v,u\in\R$
\begin{multline}
\label{eq:set of conditions 2}
\bigg(\sign(v-u)[f(x,v)-f(x,u)]-\sign(v-u)[f(y,v)-f(y,u)]\bigg)\cdot(x-y)\geq -\gamma|v-u||x-y|^2.
\end{multline}
As is explained in \cite{KarlsenRisebro2000}, there exist functions which satisfy conditions \eqref{eq:set of conditions 1}, \eqref{eq:set of conditions 2}, but they are not Lipschitz with respect to the spatial variable $x$, so the conditions in Assumption \ref{ass:assumption on the flux} are not satisfied. Conversely, there are fluxes which satisfy the conditions of Assumption \ref{ass:assumption on the flux} but not the conditions \eqref{eq:set of conditions 1}, \eqref{eq:set of conditions 2}. Take, for example, the flux $f:\R\times\R\to \R$, $f(x,u)=x^2$. It is easy to check that this flux satisfies the conditions of Assumption \ref{ass:assumption on the flux}, but its derivative $f_x(x,u)=2x$ is not globally bounded, so the second condition in \eqref{eq:set of conditions 1} does not hold. 
Note also that:
\begin{enumerate}
\item Item 1 in Assumption \ref{ass:assumption on the flux} is more general than item 3) in \eqref{eq:set of conditions 1};

\item In Assumption 1, there is no Lipschitz requirement on the derivatives of the flux as in item 4) in \eqref{eq:set of conditions 1}. However, there is a continuity assumption; see 2a in Assumption \ref{ass:assumption on the flux};

\item Note that items 1) and 2) in \eqref{eq:set of conditions 1} do not imply the locally Lipschitz condition. Item 2) states that the spatial divergence is bounded, but other partial weak derivatives might not be bounded. In dimension $1$, these two items, together with item 3), imply that the flux $f$ is Lipschitz with respect to the spatial variable $x$ and the real variable $u$, which does not guarantee the locally Lipschitz property in the variable $(x,u)$. The opposite is not true as well: if $f$ is locally Lipschitz in the variable $(x,u)\in\R^d\times\R$, then the almost everywhere divergence $\diver_xf$ is only locally bounded and not globally, hence item 2) of \eqref{eq:set of conditions 1} does not hold in general.
\end{enumerate}
\end{rem}

\begin{rem}[Comparison with an alternative setting of conditions]
In \cite{KarlsenChen2005}, Karlsen and Chen proved the contraction property for entropy solutions assuming that the flux $f=f(x,u):\R^d\times \R\to \R^d$ satisfies the following hypotheses:
\begin{equation}
\label{eq:set of conditions 3}
\begin{cases}
1)\, f(\cdot,u)\in L^\infty(\R^d,\R^d)\cap W^{1,\infty}(\R^d,\R^d)\quad & u\in I, \\
2)\, f(x,\cdot)\in W^{1,\infty}(I,\R^d)\quad & x\in \R^d.
\end{cases}
\end{equation}
Here, $I$ is a fixed closed and bounded interval in $\R$. Recall that the sobolev space $W^{1,\infty}(\R^d,\R^d)$ coincides with the space of Lipschitz functions from $\R^d$ to $\R^d$. The conditions in \eqref{eq:set of conditions 3} do not imply Assumption \ref{ass:assumption on the flux} because they do not guarantee the local Lipschitz continuity in the variable $(x,u)\in \R^d\times\R$. Conversely, Assumption \ref{ass:assumption on the flux} does not imply the conditions in \eqref{eq:set of conditions 3}, as it does not ensure global boundedness or global Lipschitz continuity as required by $1)$ in \eqref{eq:set of conditions 3}.
\end{rem}

\section{Lipschitz analysis}
\label{sec:Lipschitz analysis}
In this section, we state and prove some results about Lipschitz analysis that will be used throughout the article. We present and prove two technical lemmas concerning integration by parts and differentiation under the sign of the integral involving Lipschitz functions. These lemmas are employed to establish regularity properties for the entropy flux (see Definition \ref{def:entropy pair R}).

Recall Rademacher's Theorem \cite{evans2015measure,F69}:
\begin{thm}(Rademacher's Theorem, \cite{evans2015measure,F69})
\label{thm:Rademacher's theorem}
Let 
$f:\mathbb{R}^d\to \mathbb{R}$ be a locally Lipschitz function. Then, $f$ is differentiable almost everywhere. In particular, it has partial derivatives that exist almost everywhere
and lie in the space $L^\infty_{\text{loc}}(\mathbb{R}^d)$.
\end{thm}

\begin{lem}(Differentiation under the sign of the integral for Lipschitz functions)
\label{lem:derivation under the sign of the integral for Lipschitz functions}
Let $G:\R^d\times\R\to \R^d$, $G=G(x,\omega)$, be a locally Lipschitz function, $\xi\in L^\infty_{\text{loc}}(\mathbb{R})$, and let $B\subset\R$ be a bounded Lebesgue measurable set. Then, we have, for almost every $x\in \R^d$,
\begin{equation}
\label{eq:derivation under the sign of the integral}
D_x\int_{B} \xi(\omega)G(x,\omega)d\omega=\int_{B}\xi(\omega)D_xG(x,\omega)d\omega.
\end{equation}
\end{lem}

\begin{proof}  
It suffices to prove that for every $1\leq i\leq d,i\in \mathbb{N},$ we get for almost every $x\in \R^d$
\begin{equation}
\label{eq:derivation under the integral with respect to xi}
\partial_{x_i}\int_{B} \xi(\omega)G(x,\omega)d\omega=\int_{B}\xi(\omega)\partial_{x_i}G(x,\omega)d\omega.
\end{equation}

Let $E\subset\R^d$ be any bounded Lebesgue measurable set. Since $G$ is locally Lipschitz in $\R^d\times \R$ we get by Rademacher's Theorem
for almost every $(x,\omega)\in \R^d\times\R$
\begin{equation}
\lim_{h\to 0}\left(\frac{G(x+he_i,\omega)-G(x,\omega)}{h}-\partial_{x_i}G(x,\omega)\right)=0,
\end{equation}
where $e_i:=(0,...,1,...,0)$ is the unit vector with $1$ in the $i$-place.
Let us define a family of functions 
\begin{equation}
H_h(x,\omega):=\xi(\omega)\left(\frac{G(x+he_i,\omega)-G(x,\omega)}{h}-\partial_{x_i}G(x,\omega)\right),\quad h\in \mathbb{R}\setminus \{0\}.
\end{equation}
Let $U\subset\R^d$ be an open and bounded set such that $\overline{E}\subset U$. Since $G$ is locally Lipschitz in $\R^d\times\R$, then $G$ is Lipschitz in $U\times B$. Let us denote its Lipschitz constant in this set by $L$. Note that, for sufficiently small $|h|>0$ and almost every $(x,\omega)\in E\times B$, we get 
\begin{equation}
|H_h(x,\omega)|\leq \|\xi\|_{L^\infty(B)}\left(L+|\partial_{x_i}G(x,\omega)|\right).
\end{equation}

By Dominated Convergence Theorem we get
\begin{equation}
\label{eq:changing limit and integral}
\lim_{h\to 0}\int_E\int_{B}H_h(x,\omega)d\omega dx
=\int_E\int_{B}\xi(\omega)\lim_{h\to 0}\left(\frac{G(x+he_i,\omega)-G(x,\omega)}{h}-\partial_{x_i}G(x,\omega)\right)d\omega dx=0.
\end{equation}
Since $G$ is locally Lipschitz in $\R^d\times \R$, then the function $x\mapsto \int_{B}\xi(\omega)G(x,\omega)d\omega$ is locally Lipschitz in $\R^d$. Therefore, for almost every $x\in \R^d$ we obtain by Rademacher's Theorem 
\begin{equation}
\label{eq:partial derivative for locally Lipschitz function}
\lim_{h\to 0}\left(\frac{\int_{B}\xi(\omega)G(x+he_i,\omega)d\omega-\int_{B}\xi(\omega)G(x,\omega)d\omega}{h}-\partial_{x_i}\int_{B}\xi(\omega)G(x,\omega)d\omega\right)=0.
\end{equation}
We now use the Dominated Convergence Theorem once again, with its justification provided after \eqref{eq:integral on E of derivatives vanishes}. We get by Dominated Convergence Theorem and equations \eqref{eq:changing limit and integral}, \eqref{eq:partial derivative for locally Lipschitz function} we get
\begin{multline}
\label{eq:integral on E of derivatives vanishes}
\int_E\left(\partial_{x_i}\int_{B}\xi(\omega)G(x,\omega)d\omega-\int_{B} \xi(\omega)\partial_{x_i}G(x,\omega)d\omega\right)dx
\\
=\int_E\left(\lim_{h\to 0}\frac{\int_{B}\xi(\omega)G(x+he_i,\omega)d\omega-\int_{B}\xi(\omega)G(x,\omega)d\omega}{h}-\int_{B}\xi(\omega)\partial_{x_i}G(x,\omega)d\omega\right)dx
\\
=\int_E\lim_{h\to 0}\int_{B}\xi(\omega)\left(\frac{G(x+he_i,\omega)-G(x,\omega)}{h}-\partial_{x_i}G(x,\omega)\right)d\omega dx
\\
=\lim_{h\to 0}\int_E\int_{B}\xi(\omega)\left(\frac{G(x+he_i,\omega)-G(x,\omega)}{h}-\partial_{x_i}G(x,\omega)\right)d\omega dx=0.
\end{multline}
Since $E\subset \R^d$ is an arbitrary bounded Lebesgue measurable set, we get for almost every $x\in \R^d$ the formula \eqref{eq:derivation under the integral with respect to xi}.

Let us explain the use of the Dominated Convergence Theorem in \eqref{eq:integral on E of derivatives vanishes}. Since $G$ is locally Lipschitz in $\R^d\times \R$, we obtain, by Rademacher's Theorem, that $\partial_{x_i}G \in L^\infty_{\text{loc}}(\R^d\times\R,\R^d)$, and so $\partial_{x_i}G \in L^1_{\text{loc}}(\R^d\times\R,\R^d)$. By Fubini's Theorem, we conclude that the function $x\mapsto \int_{B}\left|\partial_{x_i}G(x,\omega)\right|d\omega$ lies in the space $L^1(E)$. Let us denote
\begin{equation}
F_h(x):=\int_{B}\xi(\omega)\left(\frac{G(x+he_i,\omega)-G(x,\omega)}{h}-\partial_{x_i}G(x,\omega)\right)d\omega.
\end{equation}
Note that, by \eqref{eq:partial derivative for locally Lipschitz function}, the limit of $F_h$ as $h\to 0$ exists almost everywhere. For an open and bounded set $U\subset\R^d$ such that $\overline{E}\subset U$, let $L$ be the Lipschitz constant of $G$ on the bounded set $U\times B$. For every $h\in \R\setminus \{0\}$ such that $|h|<\dist(E,\R^d\setminus U)$ and almost every $x\in E$, we have
\begin{multline}
\label{eq:estimate for Fh}
|F_h(x)|\leq \|\xi\|_{L^\infty(B)}\int_{B}\left|\frac{G(x+he_i,\omega)-G(x,\omega)}{h}\right|d\omega+\|\xi\|_{L^\infty(B)}\int_B\left|\partial_{x_i}G(x,\omega)\right|d\omega 
\\
\leq \|\xi\|_{L^\infty(B)}\mathcal{L}^1(B)L+\|\xi\|_{L^\infty(B)}\int_B\left|\partial_{x_i}G(x,\omega)\right|d\omega.
\end{multline}
Thus, the function in the right hand side of \eqref{eq:estimate for Fh} lies in $L^1(E)$. So we can use the Dominated Convergence Theorem in the last line of \eqref{eq:integral on E of derivatives vanishes}.
\end{proof}

\begin{lem}(Integration by parts for Lipschitz functions) 
\label{lem:integration by parts for Lipschitz functions}
Let $g:\R\to \R$ and $h:\R \to \R^d$ be locally Lipschitz functions. Then, the following integration by parts formula holds for every interval $[k_0,k]\subset\R$:
\begin{equation}
\label{eq:integration by parts of locally Lipschitz functions}
\intop_{k_0}^kg(\omega)h'(\omega)d\omega=-\intop_{k_0}^kg'(\omega)h(\omega)d\omega+g(k)h(k)-g(k_0)h(k_0).
\end{equation}
\end{lem}

\begin{proof}
Since $g$ and $h$ are locally Lipschitz, their product $gh$ is also locally Lipschitz. Therefore, $gh$ is absolutely continuous on bounded closed intervals in $\R$. Therefore, 
\begin{equation}
\label{eq:usage of AC}
\begin{cases}
(gh)'(\omega)=g'(\omega)h(\omega)+g(\omega)h'(\omega)\quad \text{for almost every $\omega$ in $\R$};\\
\intop_{k_0}^k(gh)'(\omega)d\omega
=g(k)h(k)-g(k_0)h(k_0)
\end{cases}.
\end{equation}
We get \eqref{eq:integration by parts of locally Lipschitz functions} from \eqref{eq:usage of AC}. Refer to \cite{Rudin} for a proof of the Fundamental Theorem of Calculus for absolutely continuous functions. 
\end{proof}

The following proposition will enable us to define a notion of "entropy flux" for a given flux $f$. It states that the "entropy flux" $q$ possesses some regularity properties similar to those of the flux $f$ as specified in Assumption \ref{ass:assumption on the flux}:
\begin{prop}(Properties of the entropy flux)
\label{prop:properties of the entropy flux}
Let $f:\R^d\times\R\to \R^d$ be a function satisfying Assumption \ref{ass:assumption on the flux},
$\eta\in C^2(\mathbb{R})$ and $k_0\in \mathbb{R}$.
Let us set
\begin{equation}
q\left(x,k\right):=\intop_{k_{0}}^{k}\eta'\left(\omega\right)\partial_{\omega}f\left(x,\omega\right)d\omega.
\end{equation}
Then, for every $(x,k)\in \R^d\times \R$, we have the following integration by parts representation of the function $q$:
\begin{equation}
\label{eq:identity for integration by parts in formulation}
q(x,k)=-\intop_{k_0}^k\eta''(\omega)f(x,\omega)d\omega+\eta'(k)f(x,k)-\eta'(k_0)f(x,k_0).
\end{equation}
In particular, we get:
\begin{enumerate}
\item $q:\R^d\times\R\to \R^d$
is a locally Lipschitz function (in particular, it is defined everywhere in $\R^d\times\R$).
\item For almost every 
$x\in \R^d$
the function
$k\mapsto  D_xq(x,k)$ is continuous.

\item The family of functions $\{q(\cdot,k)\}_{k\in\R}$ is uniformly differentiable almost everywhere as defined in Remark \ref{rem:Uniform differentiability almost everywhere}.
\end{enumerate}
\end{prop}

\begin{proof}
Let $(x,k)\in \R^d\times\R$. Since $f$ is locally Lipschitz, the function $\omega \mapsto f(x,\omega)$ is locally Lipschitz in $\R$. Therefore, it is differentiable almost everywhere and $\partial_{\omega}f(x,\cdot)\in L^\infty_{\text{loc}}(\R,\R^d)$. Consequently, $q$ is defined everywhere in $\R^d\times\R$ with values in $\R^d$. Let us denote $g(\omega):=\eta'(\omega)$ and $h(\omega):=f(x,\omega)$. Note that $g$ and $h$ are locally Lipschitz functions. Therefore, we get, by integration by parts formula (Lemma \ref{lem:integration by parts for Lipschitz functions})
\begin{equation}
\label{eq:identity for integration by parts}
q(x,k)=\intop_{k_0}^k\eta'(\omega)\partial_\omega f(x,\omega)d\omega
=-\intop_{k_0}^k\eta''(\omega)f(x,\omega)d\omega+\eta'(k)f(x,k)-\eta'(k_0)f(x,k_0).
\end{equation}
Since $(x,k)\in \R^d\times \R$ was arbitrary, the formula \eqref{eq:identity for integration by parts} is valid for every $(x,k)\in \R^d\times \R$, which proves \eqref{eq:identity for integration by parts in formulation}. 

\begin{enumerate}
\item Note that since $f$ is locally Lipschitz in $\R^d\times\R$, the function $G(x,k):=\intop_{k_0}^k\eta''(\omega)f(x,\omega)d\omega$ is also locally Lipschitz. Indeed, let $D\subset\R^d$ be any compact set, and $[a,b]\subset\R$ be a closed and bounded interval. Denote $S:=D\times[a,b]$. Let $L$ be the Lipschitz constant of $f$ on $D\times\left[\min\{k_0,a\},\max\{k_0,b\}\right]$. For every $(x,k),(x',k')\in S$, we have 
\begin{multline}
|G(x,k)-G(x',k')|=\left|\intop_{k_0}^k\eta''(\omega)f(x,\omega)d\omega-\intop_{k_0}^{k'}\eta''(\omega)f(x',\omega)d\omega\right|
\\
=\left|\intop_{k_0}^k\eta''(\omega)f(x,\omega)d\omega-\left(\intop_{k_0}^{k}\eta''(\omega)f(x',\omega)d\omega+\intop_{k}^{k'}\eta''(\omega)f(x',\omega)d\omega\right)\right|
\\
\leq \intop_{\min\{k_0,k\}}^{\max\{k_0,k\}}|\eta''(\omega)|\left|f(x,\omega)-f(x',\omega)\right|d\omega+\intop_{\min\{k,k'\}}^{\max\{k,k'\}}|\eta''(\omega)||f(x',\omega)|d\omega
\\
\leq L\left(\intop_{\min\{k_0,a\}}^{\max\{k_0,b\}}|\eta''(\omega)|d\omega\right)|x-x'|+\|\eta''\|_{L^\infty([a,b])}\|f\|_{L^\infty(S)}|k-k'|\leq C|(x,k)-(x',k')|,
\end{multline}
where $C$ is some constant independent of $x,x',k,k'$. It proves that $G$ is locally Lipschitz.

By identity $\eqref{eq:identity for integration by parts}$, we see that $q$ is a locally Lipschitz function in $\R^d\times\R$ as a product and sum of locally Lipschitz functions in $\R^d\times\R$. It completes the proof of part 1 of the proposition.

\item By \eqref{eq:identity for integration by parts} and differentiation under the sign of the integral (Lemma \ref{lem:derivation under the sign of the integral for Lipschitz functions}), we get for every $k\in \R$
\begin{equation}
\label{eq:description of partial derivative of q_n by integration by party}
D_xq(x,k)
=-\intop_{k_0}^k\eta''(\omega)D_xf(x,\omega)d\omega+\eta'(k)D_xf(x,k)-\eta'(k_0)D_xf(x,k_0),
\end{equation}
for almost every $x\in\R^d$.
By item 2a of Assumption \ref{ass:assumption on the flux}, the functions $k\mapsto \intop_{k_0}^k\eta''(\omega)D_xf(x,\omega)d\omega, k\mapsto\eta'(k)D_xf(x,k)$
are continuous for almost every
$x\in \R^d$. Therefore, the function $k\mapsto D_xq(x,k)$ is a continuous function for almost every $x\in \R^d.$

\item Assume the family $\{f(\cdot,k)\}_{k\in \R}$ is uniformly differentiable at $x_0\in \R^d$. Let $C\subset\R$ be any compact set which contains the interval between $k$ and $k_0$. Then, for every $(x,k)\in \R^d\times \R$, $x\neq x_0$, we get  
\begin{multline}
\frac{\left|q(x,k)-q(x_0,k)-\left(\intop_{k_0}^k\eta''(\omega)D_xf(x_0,\omega)d\omega\right)(x-x_0)\right|}{|x-x_0|}
\\
=\frac{\left|\intop_{k_0}^k\eta''(\omega)f(x,\omega)d\omega-\intop_{k_0}^k\eta''(\omega)f(x_0,\omega)d\omega-\left(\intop_{k_0}^k\eta''(\omega)D_xf(x_0,\omega)d\omega\right)(x-x_0)\right|}{|x-x_0|}
\\
\leq \int_{C}\eta''(\omega)\frac{\left|f(x,\omega)-f(x_0,\omega)-D_xf(x_0,\omega)(x-x_0)\right|}{|x-x_0|}d\omega
\\
\leq \mathcal{L}^1(C)\sup_{\omega\in C}\eta''(\omega)\sup_{\omega\in C}\frac{\left|f(x,\omega)-f(x_0,\omega)-D_xf(x_0,\omega)(x-x_0)\right|}{|x-x_0|}.
\end{multline}
By assumption $2b$ on the flux $f$, we get that the family $\{q(\cdot,k)\}_{k\in\R}$ is uniformly differentiable at $x_0$. Hence, by property 2a of the flux $f$ in Assumption \ref{ass:assumption on the flux}, we get the uniform differentiability almost everywhere of the family $\{q(\cdot,k)\}_{k\in\R}$.
\end{enumerate}
\end{proof}

\section{Entropy solutions for scalar conservation laws with Lipschitz continuous flux}
\label{sec:scalar conservation laws}

In this section, we introduce the main concepts of this work: scalar conservation laws and the notion of entropy solutions for such laws. 
 
For every $k\in\mathbb{R}$, we assume the existence of a function $f(\cdot,k):\mathbb{R}^d\to\mathbb{R}^d$, where $f = f(x,k)$. Let $u_{0}:\mathbb{R}^d\to\mathbb{R}$ be a function. Denote $I := (0,\infty)$.
Consider the Cauchy problem 
\begin{equation}
\label{eq: The Cauchy problem for the scalar conservation law}
\begin{cases}
\partial_{t}u\left(x,t\right)+\diver f\left(x,u\left(x,t\right)\right)=0 & ,(x,t)\in\mathbb{R}^d\times I\\
u\left(x,0\right)=u_{0}\left(x\right) & ,x\in\mathbb{R}^d
\end{cases}.
\end{equation}
Here, $u:\mathbb{R}^d\times I\to\mathbb{R}$ is the unknown function. The function $u_0$ represents the initial data. We refer to the equation $\partial_{t}u(x,t) + \text{div}\,f(x,u(x,t)) = 0$ as a \textbf{scalar conservation law}, and we call the function $f$ the \textbf{flux} of the law. The divergence operator $\diver$ is taken with respect to the spatial variable $x$. 

By Proposition \ref{prop:properties of the entropy flux} we can define a well-defined notion of an "entropy flux" as follows: 
\begin{defn}
\emph{\label{def:entropy pair R}(Entropy pair)}
Let $f$ be a flux as in Assumption \ref{ass:assumption on the flux}. For $C^{2}$ convex function $\eta:\mathbb{R}\to \mathbb{R}$
and $k_{0}\in\mathbb{R}$ let us define a
function
\begin{equation}
\label{eq:definition of the entropy filed q}
q\left(x,k\right):=\intop_{k_{0}}^{k}\eta'\left(\omega\right)\partial_{\omega}f\left(x,\omega\right)d\omega,\quad q:\R^d\times\R\to \R^d.
\end{equation}
We refer to the function $\eta$ as the \textbf{entropy} and the function $q$ as the \textbf{entropy flux}. The combination $\left(\eta, q\right)$ is termed an \textbf{entropy pair}.
\end{defn}

\begin{rem}(Uniqueness of derivatives of the entropy flux)
For a given $\eta$ as in Definition $\ref{def:entropy pair R}$, an entropy flux $q$ for which $(\eta, q)$ is an entropy pair is not unique, in general. Different choices of the number $k_0$ result in different entropy fluxes. More precisely, for $k_0, k_1$, we have
\begin{equation}
\intop_{k_{0}}^{k}\eta'\left(\omega\right)\partial_{\omega}f\left(x,\omega\right)d\omega=\intop_{k_{0}}^{k_1}\eta'\left(\omega\right)\partial_{\omega}f\left(x,\omega\right)d\omega+\intop_{k_1}^{k}\eta'\left(\omega\right)\partial_{\omega}f\left(x,\omega\right)d\omega.
\end{equation}  
However, by item 1 of Proposition \ref{prop:properties of the entropy flux} and Rademacher's Theorem we get for almost every
$(x,k)\in \R^d\times\R$
\begin{equation}
\partial_k\intop_{k_{0}}^{k}\eta'\left(\omega\right)\partial_{\omega}f\left(x,\omega\right)d\omega=\partial_k\intop_{k_1}^{k}\eta'\left(\omega\right)\partial_{\omega}f\left(x,\omega\right)d\omega.
\end{equation}
Which means that two different entropy fluxes for a given entropy $\eta$ have the same derivative with respect to the variable $k$ almost everywhere in $\R^d \times \R$.
\end{rem}

We now provide the definition of entropy solutions. Some notation involved in this definition will be explained in Remark \ref{rem:notation in the entropy inequality}.
\begin{defn}
\label{def:(entropy-solution)}(Entropy solution)
For $u_{0}\in L^{\infty}\left(\mathbb{R}^d\right)$ we say that
a function $u\in L^{\infty}\left(\mathbb{R}^d\times I\right)$, where $I=(0,\infty)$ is the time interval,
is an \textbf{entropy solution} for $\eqref{eq: The Cauchy problem for the scalar conservation law}$
if and only if the following two conditions hold:
\begin{enumerate}
\item For every entropy pair $\left(\eta,q\right)$ and test function
$0\leq\varphi\in \operatorname{Lip}_{c}\left(\mathbb{R}^d\times I\right)$\footnote{The space $\operatorname{Lip}_{c}\left(\mathbb{R}^d\times I\right)$ consists of Lipschitz functions in the variable $(x,t)\in \mathbb{R}^d\times I$ with compact
support.}
it follows that
\begin{multline}
\label{eq:entropy inequality}
\intop_{I}\intop_{\mathbb{R}^d}\bigg[\left(\partial_{t}\varphi\left(x,t\right)\right)\eta\left(u\left(x,t\right)\right)+\varphi\left(x,t\right)\bigg(\diver_xq\left(x,u\left(x,t\right)\right)-\eta'\left(u\left(x,t\right)\right)\diver_x f\left(x,u\left(x,t\right)\right)\bigg)
\\
+\left(\nabla_{x}\varphi\left(x,t\right)\right)\cdot q\left(x,u\left(x,t\right)\right)\bigg]dxdt
\geq0.
\end{multline}
Omitting variables in the above inequality gives
\begin{equation}
\intop_{I}\intop_{\mathbb{R}^d}\bigg[\partial_{t}\varphi\,\eta(u)+\varphi\bigg(\diver_xq\left(x,u\right)-\eta'(u)\diver_x f\left(x,u\right)\bigg)+\nabla_{x}\varphi\cdot q\left(x,u\right)\bigg]dxdt\geq0.
\end{equation}
We call the last inequality an
\textbf{entropy inequality}.

\item There exists a set $\Psi\subset I$ such that $\mathcal{L}^1(\Psi)=0$ and for every ball $B_R(0)=\Set{x\in\R^N}[|x|<R<\infty]$, we have
\begin{equation}
\label{eq:continuity from the right at 0}
\lim_{\underset{t\in I\setminus\Psi}{t\to 0^+}}\int_{B_R(0)}|u(x,t)-u_0(x)|dx=0.
\end{equation}
\end{enumerate}
\end{defn}

\begin{rem}(Remarks about the definition of entropy solutions)
\label{rem:notation in the entropy inequality}
\begin{enumerate}

\item Note that by Rademacher's Theorem, the test function $\varphi$ has locally bounded partial derivatives $\partial_t\varphi,\nabla_x \varphi$. 

\item By Definition 
$\ref{def:entropy pair R},$
the entropy
$\eta:\R \to \R$
is differentiable and we denote its derivative by $\eta'$ and 
in the entropy inequality 
$\eqref{eq:entropy inequality}$
the amount 
$\eta'(u(x,t))$
is the derivative of 
$\eta$
at the point 
$u(x,t).$

\item We denote by
$\diver_xf$
the divergence of the flux 
$f:\R^d\times\R\to \R^d,f=f(x,k),$
with respect to
the explicit dependence on the variable 
$x$. In the entropy inequality \eqref{eq:entropy inequality}, 
$\diver_x f(x,u(x,t))=\diver f(x,k)|_{k=u(x,t)}$ and $\diver_x q(x,u(x,t))=\diver q(x,k)|_{k=u(x,t)}$. 

\item Notice that according to items 1,2a of Assumption \ref{ass:assumption on the flux}, Proposition \ref{prop:properties of the entropy flux} and $u\in L^\infty(\R^d\times I)$, we get that the functions $\diver_x f(x,u(x,t))$, $q(x,u(x,t))$, $\diver_xq(x,u(x,t))$ are defined for almost every $(x,t)\in \R^d\times I$ and are locally bounded. 
Note that property 2a is required to ensure that the functions  
\[
(x,t) \mapsto \diver_x f(x,u(x,t)), \quad (x,t) \mapsto \diver_x q(x,u(x,t))
\]  
are defined almost everywhere in \(\mathbb{R}^d \times I\).

\item In the entropy inequality \eqref{eq:entropy inequality}, the term $\nabla_{x}\varphi\cdot q\left(x,u\right)$ stands for the inner product between the vectors $\nabla_{x}\varphi(x,t),\,q\left(x,u(x,t)\right)\in \R^d$.

\item Note that the set $\Psi$ in the second condition of Definition \ref{def:(entropy-solution)} is independent of the choice of the ball $B_R(0)$.
\end{enumerate}
\end{rem}

\textbf{Discussion about the entropy inequality}\\
Let us discuss the entropy inequality (\ref{eq:entropy inequality}).
Let us assume the flux $f$ is $C^1$ and the existence of a $C^1$ solution $u=u\left(x,t\right)$ for $\eqref{eq: The Cauchy problem for the scalar conservation law}$, and let $\left(\eta,q\right)$
be an entropy pair as in Definition \ref{def:entropy pair R}. From the scalar conversation law we get
\begin{equation}
\label{eq:using the scalar law}
\partial_t u=-\diver f(x,u)=-\diver_x f\left(x,u\right)-\partial_{k}f(x,k)|_ {k=u(x,t)}\cdot \nabla_{x}u,
\end{equation}
and by the definition of entropy pair (Definition \ref{def:entropy pair R}), we get
\begin{equation}
\label{eq:derivative of the entropy flux}
\partial_{k}q(x,k)|_{k=u(x,t)}=\eta'\left(u\right)\partial_{k}f(x,k)|_{k=u(x,t)}.
\end{equation} 
Therefore, by \eqref{eq:using the scalar law} and \eqref{eq:derivative of the entropy flux} we get
\begin{multline}
\partial_{t}\eta\left(u\right)+\diver q\left(x,u\right)=\eta'\left(u\right)\partial_{t}u+\diver_xq\left(x,u\right)+\partial_{k}q(x,k)|_{k=u(x,t)}\cdot \nabla_{x}u
\\
=-\eta'\left(u\right)\bigg(\diver_x f\left(x,u\right)+\partial_{k}f(x,k)|_{k=u(x,t)}\cdot \nabla_{x}u\bigg)+\diver_xq\left(x,u\right)+\eta'\left(u\right)\partial_{k}f(x,k)|_{k=u(x,t)}\cdot\nabla_{x}u
\\
=-\eta'\left(u\right)\diver_x f\left(x,u\right)+\diver_xq\left(x,u\right).
\end{multline}
Thus,
\begin{equation}
\partial_{t}\eta\left(u\right)+\diver q\left(x,u\right)+\eta'\left(u\right)\diver_x f\left(x,u\right)-\diver_xq\left(x,u\right)=0.
\end{equation}
In particular,
\begin{equation}
\partial_{t}\eta\left(u\right)+\diver q\left(x,u\right)+\eta'\left(u\right)\diver_x f\left(x,u\right)-\diver_xq\left(x,u\right)\leq0.
\end{equation}
By multiplying the last inequality by test function $\varphi$ as in the definition of the entropy solution (Definition
\ref{def:(entropy-solution)}) and using integration by parts we
get the entropy inequality (\ref{eq:entropy inequality}). 

We give now an approximation lemma which tells us that we can take
in the inequality (\ref{eq:entropy inequality}) an "entropy pair" $(\eta,q)$ such that the entropy $\eta$ is not $C^2$.  
\begin{lem}(Validity of the entropy inequality for non-regular entropy $\eta$) 
\label{lem: approximation lemma R}
Let $u$ be an entropy solution, as defined in Definition \ref{def:(entropy-solution)}, for the Cauchy problem \eqref{eq: The Cauchy problem for the scalar conservation law}. Fix $k_{0}\in\R$ and define functions $\eta:\R\to\R$ and $q:\R^d\times\R\to\R^d$ as follows:
\begin{equation}
\label{eq:definition of eta}
\eta\left(k\right):=\left|k-k_{0}\right|,
\end{equation}
\begin{equation}
\label{eq:definition of q}
q\left(x,k\right):=\sign(k-k_{0})\left(f\left(x,k\right)-f\left(x,k_{0}\right)\right).
\end{equation}
Here, $f$ is a function satisfies items $1,2a$ of Assumption \ref{ass:assumption on the flux}. Then, the pair $\left(\eta,q\right)$ satisfies the entropy inequality (\ref{eq:entropy inequality}) together with any $0\le\varphi\in \operatorname{Lip}_{c}\left(\R^d\times I\right)$, $I=(0,\infty)$. The pair $(\eta,q)$ is called \textbf{Kruzkov's entropy pair}.
\end{lem}

\begin{proof}
The derivative of $\eta$ is given by $\eta'(k)=\sign(k-k_0)$ for every $k\neq k_0$ and we define $\eta'(k_0):=0$.
Consider a sequence of $C^{2}$ convex functions $\eta_{n}:\R\to\R$
such that
\begin{equation}
\label{eq:properties for eta n}
\begin{cases}
\eta_{n}\to\eta\quad
\text{in}\quad  L_{\text{loc}}^{\infty}\left(\R\right),\\
\eta_{n}'(k)\to \sign\left(k-k_0\right), \forall k\in \R,\\
\eta_{n}'(k_0)=0,\quad \forall n\in \N,\\
\exists 0<C<\infty \Longrightarrow \left\Vert \eta_{n}'\right\Vert _{L^{\infty}\left(\R\right)}\leq C,\forall n\in \N
\end{cases}.
\end{equation}
One can take, for example,
$\eta_{n}\left(k\right):=\sqrt{\left(k-k_{0}\right)^{2}+\frac{1}{n}}$.
For each $\eta_{n}$ define
\begin{equation}
q_{n}\left(x,k\right):=\int_{k_{0}}^{k}\eta_{n}'\left(w\right)\partial_{w}f\left(x,w\right)dw.
\end{equation}
Hence, $\left(\eta_{n},q_{n}\right)$
is an entropy pair. Since $u$
is an entropy solution, we obtain for
$0\le\varphi\in \operatorname{Lip}_{c}\left(\R^d\times I\right)$:

\begin{equation}
\label{eq:entropy inequlity for convex pair (eta_n,q_n)}
\intop_{I}\intop_{\mathbb{R}^d}\Big[\partial_{t}\varphi\,\eta_{n}\left(u\right)+\varphi\,\bigg(\diver_xq_{n}\left(x,u\right)-\eta_n'\left(u\right)\diver_x f\left(x,u\right)\bigg)+\nabla_{x}\varphi\cdot q_{n}\left(x,u\right)\Big]dxdt\geq 0.
\end{equation}
We intend to take the limit as $n\to \infty$ in \eqref{eq:entropy inequlity for convex pair (eta_n,q_n)}.

In the following, we use the following notations. We denote by $\Supp(\varphi)$ the support of $\varphi$; since $\varphi\in \operatorname{Lip}_{c}(\R^d\times I)$, then $\Supp(\varphi)$ is a compact set in $\R^d\times I$. We denote by $\image u$ the image of $u$; since $u\in L^\infty(\R^d\times I)$, then $\image u$ is contained in a compact set in $\R$ outside a set of measure zero. We denote by $P_x:\R^d\times \R\to \R^d$, $P_x(x,k):=x$, the projection onto $\R^d$; $P_x$ takes compact sets in $\R^d\times \R$ to compact sets in $\R^d$.

We prove that:
\begin{equation}
\label{eq:limit of double integral involving eta_n}
\lim_{n\to \infty}\intop_{I}\intop_{\mathbb{R}^d}\partial_{t}\varphi\,\eta_{n}(u)dxdt=\intop_{I}\intop_{\mathbb{R}^d}\partial_{t}\varphi\,\eta(u)dxdt.
\end{equation}
Note that
\begin{equation}
\label{eq:limit of double integral involving eta_n 1}
\intop_{I}\intop_{\mathbb{R}^d}|\partial_{t}\varphi|\,\left|\eta_{n}(u)-\eta(u)\right|dxdt\leq \mathcal{L}^{d+1}(\Supp(\varphi))\|\partial_t\varphi\|_{L^\infty(\Supp(\varphi))}\|\eta_n-\eta\|_{L^\infty(\image u)}.
\end{equation}

Since $\eta_{n}\to\eta$ in $L_{\text{loc}}^{\infty}(\R)$, we obtain \eqref{eq:limit of double integral involving eta_n}. 

We prove that:
\begin{equation}
\label{eq:a limit of double integral involving eta_n' and q_n-part 1}
\lim_{n\to \infty}\intop_{I}\intop_{\mathbb{R}^d}\varphi\,\eta_n'\left(u\right)\diver_x f\left(x,u\right)dxdt
=\intop_{I}\intop_{\mathbb{R}^d}\varphi\,\eta'\left(u\right)\diver_x f\left(x,u\right)dxdt.
\end{equation}
According to Remark \ref{rem:notation in the entropy inequality}, the function $(x,t)\mapsto \diver_xf(x,u(x,t))$ lies in the space $L^\infty_{\text{loc}}(\R^d\times I)$.
According to the convergence $\eta_{n}'(k)\to \eta'(k):=\sign(k-k_0)$ for all $k\in \R$ and the existence of a constant $C$ such that $\left\Vert \eta_{n}'\right\Vert _{L^{\infty}(\R)}\leq C$ for every $n\in\N$, we can apply the Dominated Convergence Theorem to get \eqref{eq:a limit of double integral involving eta_n' and q_n-part 1}.

We now prove that:
\begin{equation}
\label{eq:a limit of double integral involving eta_n' and q_n-part 2}
\lim_{n\to \infty}\intop_{I}\intop_{\mathbb{R}^d}\nabla_{x}\varphi\cdot q_{n}\left(x,u\right)dxdt
=\intop_{I}\intop_{\mathbb{R}^d}\nabla_{x}\varphi\cdot q\left(x,u\right)dxdt.
\end{equation}

We first prove that, for almost every $(x,t)\in\R^d\times I$ 
\begin{equation}
\label{eq:limit of qn in u to q in u}
\lim_{n\to \infty}q_{n}\left(x,u(x,t)\right)=q(x,u(x,t)).
\end{equation}

Since $f$ is locally Lipschitz, for every $x\in \R^d$, the function $f_x(w):=f(x,w)$ is locally Lipschitz in $\R$. Therefore, $\partial_wf(x,\cdot)=f'_x\in L^\infty_{\text{loc}}(\R)$. By the Dominated Convergence Theorem and the Fundamental Theorem of Calculus, we have for every $k\in\R$ and every $x\in \R^d$
\begin{multline}
\lim_{n\to \infty}q_{n}\left(x,k\right)=\intop_{k_{0}}^{k}\lim_{n\to \infty}\eta_{n}'\left(w\right)\partial_{w}f\left(x,w\right)dw=\intop_{k_{0}}^{k}\sign(w-k_0)\partial_{w}f\left(x,w\right)dw
\\
=\sign(k-k_0)\intop_{k_{0}}^{k}\partial_{w}f(x,w)dw=\sign(k-k_0)\left(f(x,k)-f(x,k_0)\right)=q(x,k).
\end{multline}
It proves \eqref{eq:limit of qn in u to q in u}.
Now, note that for every $n\in\N$ and $k\in\R$, we get the estimate
\begin{equation}
\label{eq:upper bound on a sequence of entropy fluxes}
|q_n(x,k)|\leq C\int_{\min\{k_{0},k\}}^{\max\{k_{0},k\}}|\partial_{w}f\left(x,w\right)|dw.
\end{equation}
Since $f$ is locally Lipschitz in $\R^d\times\R$, we get by Rademacher's Theorem that the function 
\begin{equation}
\label{eq:local integrability of F}
F(x,t):=\int_{\min\{k_{0},u(x,t)\}}^{\max\{k_{0},u(x,t)\}}|\partial_{w}f\left(x,w\right)|dw
\end{equation}
lies in the space $L^1_{\text{loc}}(\R^d\times I)$. Therefore, by \eqref{eq:limit of qn in u to q in u}, \eqref{eq:upper bound on a sequence of entropy fluxes}, \eqref{eq:local integrability of F} and Dominated Convergence Theorem, we obtain \eqref{eq:a limit of double integral involving eta_n' and q_n-part 2}.

We now prove that:
\begin{equation}
\label{eq:a limit of double integral involving spatial derivative of q_n}
\lim_{n\to\infty}\intop_{I}\intop_{\R^d}\varphi\,\diver_xq_{n}(x,u)dxdt=\intop_{I}\intop_{\R^d}\varphi\,\diver_xq(x,u)dxdt.
\end{equation}

First we prove that for almost every $(x,t)\in\R^d\times I$ we get
\begin{equation}
\label{eq:limit of spatial derivative of q_n}
\lim_{n\to \infty}\diver_xq_n(x,u(x,t))=\diver_xq(x,u(x,t)).
\end{equation}
For this purpose, let us prove the existence of a set $\Theta\subset\R^d$, $\mathcal{L}^d(\Theta)=0$, such that for every $x\in \R^d\setminus\Theta$ and $k\in \R$, we have
\begin{equation}
\label{eq:limit of spatial derivative of q_n with k}
\lim_{n\to \infty}\diver q_n(x,k)=\diver q(x,k).
\end{equation}
From \eqref{eq:limit of spatial derivative of q_n with k}, we deduce \eqref{eq:limit of spatial derivative of q_n}
because for every $(x,t)\in \R^d\times I$ such that $(x,t)$ lies in the domain of $u$ and $(x,t)\notin \Theta\times I$, we obtain \eqref{eq:limit of spatial derivative of q_n} for $x$ and $k=u(x,t)$.
Let $\Theta$ be as in item 2a of Assumption \ref{ass:assumption on the flux}. For all 
$k\in \R$
and for every $x\in \R^d\setminus\Theta$ 
we obtain
\begin{multline}
\label{eq:description of partial derivative of q_n by integration by party}
\diver q_n(x,k)=
\diver \left[\intop_{k_0}^k\eta'_n(w)\partial_w f(x,w)dw\right]
\\
=\diver \left[-\intop_{k_0}^k\eta''_n(w) f(x,w)dw+\eta_n'(k)f(x,k)-\eta_n'(k_0)f(x,k_0)\right]
\\
=-\intop_{k_0}^k\eta''_n(w) \diver f(x,w)dw+\eta_n'(k)\diver f(x,k).
\end{multline}
In the last equality of \eqref{eq:description of partial derivative of q_n by integration by party} we use $\eta_n'(k_0)=0$. The integration by parts and the differentiation under the sign of the integral follows from Lemmas \ref{lem:derivation under the sign of the integral for Lipschitz functions}, \ref{lem:integration by parts for Lipschitz functions}. Let us analyse the integral on the right hand side of \eqref{eq:description of partial derivative of q_n by integration by party}. More precisely, we show that
\begin{equation}
\label{eq:limit result involved eta''}
\lim_{n\to \infty}\intop_{k_0}^k\eta''_n(w) \diver f(x,w)dw=\sign(k-k_0)\diver f(x,k_0).
\end{equation}

 Let $k\in \R$. Let $x\in \R^d$ be such that $w\longmapsto\diver f(x,w)$ is continuous (almost every $x\in \R^d$ has this property according to item 2a of Assumption \ref{ass:assumption on the flux}). So it is continuous at $k_0$. Let $\epsilon\in(0,\infty).$ There exists $\delta\in(0,\infty)$ such that $|\diver f(x,w)-\diver f(x,k_0)|<\epsilon$ whenever $|w-k_0|<\delta$. Assume that $k>k_0$. It follows that
\begin{multline}
\label{eq:formula1}
\intop_{k_0}^k\eta''_n(w) \diver  f(x,w)dw=\intop_{k_0}^k\eta''_n(w) \diver f(x,k_0)dw-\intop_{k_0}^k\eta''_n(w) \left(\diver f(x,k_0)-\diver f(x,w)\right)dw
\\
=\diver f(x,k_0)\left(\eta'_n(k)-\eta'_n(k_0)\right)
-\intop_{k_0+\delta}^k\eta''_n(w) \left(\diver f(x,k_0)-\diver f(x,w)\right)dw
\\
-\intop_{k_0}^{k_0+\delta}\eta''_n(w) \left(\diver f(x,k_0)-\diver f(x,w)\right)dw.
\end{multline}
Observe that since 
$\eta_n$ 
is convex we know that 
$\eta_n''$ 
is a nonnegative function. By the Fundamental Theorem of Calculus we get  
\begin{multline}
\label{eq:formula2}
\limsup_{n\to \infty}\intop_{k_0+\delta}^{k}\eta''_n(w) |\diver f(x,k_0)-\diver f(x,w)|dw
\\
\leq  \sup_{w\in [k_0,k]}|\diver f(x,k_0)-\diver f(x,w)|\lim_{n\to \infty}\left(\eta'_n(k)-\eta'_n(k_0+\delta)\right)=0.
\end{multline}
In addition,
\begin{equation}
\label{eq:formula3}
\limsup_{n\to \infty}\intop_{k_0}^{k_0+\delta}\eta''_n(w) |\diver f(x,k_0)-\diver f(x,w)|dw\leq \epsilon\lim_{n\to \infty}\left(\eta'_n(k_0+\delta)-\eta'_n(k_0) \right)=\epsilon.
\end{equation}
Therefore, we get \eqref{eq:limit result involved eta''} from \eqref{eq:formula1}, \eqref{eq:formula2} and \eqref{eq:formula3}.

Taking the limit as $n\to \infty$ in $\eqref{eq:description of partial derivative of q_n by integration by party}$, and taking into account $\eqref{eq:limit result involved eta''}$, we obtain \eqref{eq:limit of spatial derivative of q_n with k}.
The case $k<k_0$ is similar. It completes the proof of \eqref{eq:limit of spatial derivative of q_n}. 

Next, according to
$\eqref{eq:description of partial derivative of q_n by integration by party}$, 
we get for almost every 
$(x,t)\in \R^d\times I$
\begin{align}
\label{eq:eq1 for uniform boundedness of q_n}
\diver_xq_n(x,u(x,t))=
-\intop_{k_0}^{u(x,t)}\eta''_n(w) \diver f(x,w)dw+\eta_n'(u(x,t))\diver_x f(x,u(x,t)).
\end{align}
Notice that, since $u$ is bounded and $f$ is locally Lipschitz, the sequence of functions $(x,t)\longmapsto \diver_xq_n(x,u(x,t))$ is uniformly bounded on compact sets in $\R^d\times I$, which allows us to use the Dominated Convergence Theorem to get \eqref{eq:a limit of double integral involving spatial derivative of q_n}. 

Taking the limit as $n\to \infty$ in
$\eqref{eq:entropy inequlity for convex pair (eta_n,q_n)}$ and taking into account 
\eqref{eq:limit of double integral involving eta_n},
\eqref{eq:a limit of double integral involving eta_n' and q_n-part 1},
\eqref{eq:a limit of double integral involving eta_n' and q_n-part 2}
and \eqref{eq:a limit of double integral involving spatial derivative of q_n}, we see that the pair $(\eta,q)$, as defined in \eqref{eq:definition of eta},\eqref{eq:definition of q}, satisfies the entropy inequality $\eqref{eq:entropy inequality}.$   
\end{proof}

\section{$L^1$-contraction property}
\label{sec:L1-contraction property}
In this section, we prove several contraction properties.  

Recall the Lebesgue Differentiation Theorem \cite{evans2015measure}:
\begin{thm}(Lebesgue Differentiation Theorem)
\label{thm:Lebesgue differentiation theorem}
Let $h\in L^1_{\text{loc}}(\R^d)$. Then,  almost every point in $\R^d$ is a Lebesgue point of $h$, i.e. a point $x\in\mathbb{R}^{d}$
such that
\begin{equation}
\lim_{\epsilon\to 0^+}\fintop_{B_{\epsilon}\left(x\right)}\left|h\left(x\right)-h\left(y\right)\right|dy:=\lim_{\epsilon\to 0^+}\frac{1}{\mathcal{L}^d(B_{\epsilon}\left(x\right))}\intop_{B_{\epsilon}\left(x\right)}\left|h\left(x\right)-h\left(y\right)\right|dy=0.
\end{equation}
Here $B_{\epsilon}(x)$
is a closed ball in $\mathbb{R}^{d}$ centred at $x$ with radius $\epsilon$.
\end{thm}

Recall that the \textbf{$d-$dimensional standard
mollifier kernel $\rho_{\epsilon}$} on $\R^{d}$ is given by
\begin{equation}
\rho_{\epsilon}\left(x\right):=\frac{1}{\epsilon^{d}}\rho\left(\frac{x}{\epsilon}\right),\quad \rho\left(x\right)=\begin{cases}
C\,\exp\left(\frac{1}{\left|x\right|^{2}-1}\right) & \left|x\right|<1\\
0 & \left|x\right|\geq1
\end{cases},
\label{eq: standard mollifier kernel}
\end{equation}
where $C>0$ is a constant such that $\intop_{\R^{d}}\rho\left(x\right)dx=1$.
\\
\\
The following lemma is the main result of this paper, in which we establish a localized contraction property, sometimes referred to as Kato's inequality.
\begin{lem}[Localized contraction property for entropy solutions]
\label{lem:Kato's inequality}
Let $u,\tilde{u}$ be entropy solutions as in definition \ref{def:(entropy-solution)}. Let $(\eta,q)$ be the entropy pair defined by 
\begin{equation}
\label{eq:definition of eta and the entropy flux in the contraction property}
\eta=\eta\left(k_{1},k_{2}\right)=\left|k_{1}-k_{2}\right|,\quad q=q\left(x,k_{1},k_{2}\right)=\sign\left(k_{1}-k_{2}\right)\left(f\left(x,k_{1}\right)-f\left(x,k_{2}\right)\right),\quad k_1,k_2\in \R,
\end{equation}
where $f$ is a flux as in Assumption \ref{ass:assumption on the flux}. Then, for every $0\leq \psi\in \operatorname{Lip}_{c}\left(\mathbb{R}^d\times I\right)$ 
\begin{equation}
\label{eq:final term, in statement of theorem}
\intop_{I}\intop_{\mathbb{R}^d}\bigg[\partial_t\psi(x,t)\eta\left(u\left(x,t\right),\tilde{u}\left(x,t\right)\right)+\nabla_x\psi(x,t)\cdot q\left(x,u\left(x,t\right),\tilde{u}\left(x,t\right)\right)\bigg]dxdt
\geq 0.
\end{equation}
\end{lem}
\begin{proof}
The proof is divided into $2$ steps.\\
\textbf{Step 1}\\
Let us fix $y\in\mathbb{R}^d$, $k_{2}\in\mathbb{R}$ and $s\in I$. We choose a pair
of functions $\left(\eta,q\right)$ defined by
\begin{equation}
\eta=\eta\left(k_{1},k_{2}\right)=\left|k_{1}-k_{2}\right|,\quad q=q\left(x,k_{1},k_{2}\right)=\sign\left(k_{1}-k_{2}\right)\left(f\left(x,k_{1}\right)-f\left(x,k_{2}\right)\right).
\end{equation}

By Lemma \ref{lem: approximation lemma R} we know that this pair
satisfies the entropy inequality (\ref{eq:entropy inequality}). Let us
choose a test function of the form

\begin{equation}
\label{eq:definition of phi}
\varphi=\varphi\left(x,t,y,s\right):=\psi(x,t)\omega_\epsilon(t-s)\rho_{\epsilon}\left(x-y\right).
\end{equation}
Here $\omega_\epsilon$ stands for the $1$-dimensional standard mollifier kernel and $\rho_{\epsilon}$ stands for the $d$-dimensional standard mollifier
kernel as in (\ref{eq: standard mollifier kernel}), $0\leq \psi\in \operatorname{Lip}_{c}\left(\mathbb{R}^d\times I\right)$. Since $u$
is an entropy solution, we get by Definition \ref{def:(entropy-solution)}
\begin{equation}
\intop_{I}\intop_{\mathbb{R}^d}\bigg[\partial_{t}\varphi\,\eta\left(u,k_{2}\right)+\varphi\,\bigg(\diver_xq\left(x,u,k_{2}\right)-\partial_{1}\eta\left(u,k_{2}\right)\diver_x f\left(x,u\right)\bigg)+\nabla_{x}\varphi\cdot q\left(x,u,k_{2}\right)\bigg]dxdt\geq 0.
\end{equation}

Here $u=u\left(x,t\right)$, and $\partial_{1}\eta$ stands for the
partial derivative of $\eta$ with respect to the first variable. We choose
$k_{2}=\tilde{u}\left(y,s\right)$. By integrating the last inequality
on $\mathbb{R}^d\times I$ with respect to $dyds$ and using Fubini's Theorem we get

\begin{multline}
\label{eq:first intergal R}
\intop_{I}\intop_{\mathbb{R}^d}\Bigg\{\intop_{I}\intop_{\mathbb{R}^d}\bigg[\partial_{t}\varphi\,\eta\left(u,\tilde{u}\right)+\varphi\bigg(\diver_xq\left(x,u,\tilde{u}\right)-\partial_{1}\eta\left(u,\tilde{u}\right)\diver_x f\left(x,u\right)\bigg)+\nabla_{x}\varphi\cdot q\left(x,u,\tilde{u}\right)\bigg]dyds\Bigg\}dxdt \geq0.
\end{multline}
In a similar manner, for fixed $x\in\mathbb{R}^d$, $k_{1}\in\mathbb{R}$, and
$t\in I$, by the assumption that $\tilde{u}$ is an entropy solution,
we have
\begin{equation}
\intop_{I}\intop_{\mathbb{R}^d}\bigg[\partial_{s}\varphi\,\eta\left(k_{1},\tilde{u}\right)+\varphi\,\bigg(\diver_yq\left(y,k_{1},\tilde{u}\right)-\partial_{2}\eta\left(k_{1},\tilde{u}\right)\diver_y f\left(y,\tilde{u}\right)\bigg)+\nabla_{y}\varphi\cdot q\left(y,k_{1},\tilde{u}\right)\bigg]dyds\geq0.
\end{equation}
Here $\tilde{u}=\tilde{u}\left(y,s\right)$ and $\partial_{2}\eta$
is the derivative of $\eta$ with respect to the second variable. We choose
$k_{1}=u\left(x,t\right)$. By integrating the last inequality on
$\mathbb{R}^d\times I$ with respect to $dxdt$ we get
\begin{multline}
\label{eq:second integral R}
\intop_{I}\intop_{\mathbb{R}^d}\Bigg\{\intop_{I}\intop_{\mathbb{R}^d}\bigg[\partial_{s}\varphi\,\eta\left(u,\tilde{u}\right)+\varphi\,\bigg(\diver_yq\left(y,u,\tilde{u}\right)-\partial_{2}\eta\left(u,\tilde{u}\right)\diver_y f\left(y,\tilde{u}\right)\bigg)+\nabla_{y}\varphi\cdot q\left(y,u,\tilde{u}\right)\bigg]dyds \Bigg\} dxdt\geq0.
\end{multline}
We add up (\ref{eq:first intergal R}) and (\ref{eq:second integral R}), and 
we get
\begin{multline}
\label{eq:third integral R}
\intop_{I}\intop_{\mathbb{R}^d}\Bigg\{\intop_{I}\intop_{\mathbb{R}^d}\bigg[\partial_{t}\varphi\,\eta\left(u,\tilde{u}\right)+\varphi\,\bigg(\diver_xq\left(x,u,\tilde{u}\right)-\partial_{1}\eta\left(u,\tilde{u}\right)\diver_x f\left(x,u\right)\bigg)+\nabla_{x}\varphi\cdot q\left(x,u,\tilde{u}\right)\bigg]dyds
\\
+\intop_{I}\intop_{\mathbb{R}^d}\bigg[\partial_{s}\varphi\,\left(u,\tilde{u}\right)+\varphi\,\bigg(\diver_yq\left(y,u,\tilde{u}\right)-\partial_{2}\eta\left(u,\tilde{u}\right)\diver_y f\left(y,\tilde{u}\right)\bigg)+\nabla_{y}\varphi\cdot q\left(y,u,\tilde{u}\right)\bigg]dyds\Bigg\} dxdt\geq0.
\end{multline}

In \eqref{eq:third integral R} we have $\varphi=\varphi(x,t,y,s)$. Note that we have in \eqref{eq:third integral R} an integral of the form $\intop_{I}\intop_{\mathbb{R}^d}\left\{ \cdot\right\} \,dxdt$, and within $\left\{ \cdot\right\}$, we observe analogous terms characterized by the interchange of roles between $x,t$ and $y,s$. 
For instance, the term $\intop_{I}\intop_{\mathbb{R}^d}\partial_{t}\varphi\,\eta\left(u,\tilde{u}\right) \,dyds$ is analogous to the term $\intop_{I}\intop_{\mathbb{R}^d}\partial_{s}\varphi\,\left(u,\tilde{u}\right) \,dyds$.

We sum up analogous terms from $\left\{ \cdot\right\}$ in the inequality (\ref{eq:third integral R}). 

By the definition of $\varphi$, we get
\begin{multline}
\label{eq:fourth integral R}
\intop_{I}\intop_{\mathbb{R}^d}\partial_{t}\varphi\,\eta\left(u,\tilde{u}\right)dyds+\intop_{I}\intop_{\mathbb{R}^d}\partial_{s}\varphi\,\eta\left(u,\tilde{u}\right)dyds
\\
=\intop_{I}\intop_{\mathbb{R}^d}\biggl(\partial_t\psi(x,t)\omega_\epsilon(t-s)+\psi(x,t)\partial_{t}\omega_\epsilon(t-s)+\psi(x,t)\partial_{s}\omega_\epsilon(t-s)\biggr)\rho_{\epsilon}\left(x-y\right)\eta\left(u,\tilde{u}\right)dyds
\\
=\partial_t\psi(x,t)\intop_{I}\intop_{\mathbb{R}^d}\omega_\epsilon(t-s)\rho_{\epsilon}\left(x-y\right)\eta\left(u,\tilde{u}\right)dyds,
\end{multline}
where in the second equality of \eqref{eq:fourth integral R} we have $\psi(x,t)\partial_{t}\omega_\epsilon(t-s)+\psi(x,t)\partial_{s}\omega_\epsilon(t-s)=0$.

Next, we sum up analogous terms that involve $\varphi$ (without derivative). Before doing that, note that

\begin{multline}
\label{eq:calculation1}
\diver_xq\left(x,u,\tilde{u}\right)-\partial_{1}\eta\left(u,\tilde{u}\right)\diver_x f\left(x,u\right)
\\
=\sign\left(u-\tilde{u}\right)\bigg(\diver_x f\left(x,u\right)-\diver_x f\left(x,\tilde{u}\right)\bigg)-\sign\left(u-\tilde{u}\right)\diver_x f\left(x,u\right)
\\
=-\sign\left(u-\tilde{u}\right)\diver_x f\left(x,\tilde{u}\right);
\end{multline}
\begin{multline}
\label{eq:calculation2}
\diver_yq\left(y,u,\tilde{u}\right)-\partial_{2}\eta\left(u,\tilde{u}\right)\diver_y f\left(y,\tilde{u}\right)
\\
=\sign\left(u-\tilde{u}\right)\bigg(\diver_y f\left(y,u\right)-\diver_y f\left(y,\tilde{u}\right)\bigg)+\sign\left(u-\tilde{u}\right)\diver_y f\left(y,\tilde{u}\right)
\\
=\sign\left(u-\tilde{u}\right)\diver_y f\left(y,u\right).
\end{multline}
Therefore, from \eqref{eq:calculation1} and \eqref{eq:calculation2}, we get
\begin{multline}
\bigg[\diver_xq\left(x,u,\tilde{u}\right)-\partial_{1}\eta\left(u,\tilde{u}\right)\diver_x f\left(x,u\right)\bigg]+\bigg[\diver_yq\left(y,u,\tilde{u}\right)-\partial_{2}\eta\left(u,\tilde{u}\right)\diver_y f\left(y,\tilde{u}\right)\bigg]
\\
=\sign\left(u-\tilde{u}\right)\big(\diver_y f\left(y,u\right)-\diver_x f\left(x,\tilde{u}\right)\big).
\end{multline}
Therefore,
\begin{multline}
\label{eq:fifth integral R}
\intop_{I}\intop_{\mathbb{R}^d}\varphi\,\bigg(\diver_xq\left(x,u,\tilde{u}\right)-\partial_{1}\eta\left(u,\tilde{u}\right)\diver_x f\left(x,u\right)\bigg)dyds
\\
+\intop_{I}\intop_{\mathbb{R}^d}\varphi\,\bigg(\diver_yq\left(y,u,\tilde{u}\right)-\partial_{2}\eta\left(u,\tilde{u}\right)\diver_y f\left(y,\tilde{u}\right)\bigg)dyds
\\
=\psi(x,t)\intop_{I}\intop_{\mathbb{R}^d}\omega_\epsilon(t-s)\rho_{\epsilon}\left(x-y\right)\sign\left(u-\tilde{u}\right)\bigg(\diver_y f\left(y,u\right)-\diver_x f\left(x,\tilde{u}\right)\bigg)dyds.
\end{multline}
Next, we add up analogous terms which contain $\nabla_x\varphi,\nabla_y\varphi$.
We get
\begin{multline}
\label{eq:sixth integral R}
\intop_{I}\intop_{\mathbb{R}^d}\nabla_{x}\varphi\cdot q\left(x,u,\tilde{u}\right)dyds+\intop_{I}\intop_{\mathbb{R}^d}\nabla_{y}\varphi\cdot q\left(y,u,\tilde{u}\right)dyds
\\
=\intop_{I}\intop_{\mathbb{R}^d}\bigg(\nabla_x\psi(x,t)\omega_\epsilon(t-s)\rho_{\epsilon}\left(x-y\right)+\psi(x,t)\omega_\epsilon(t-s)\nabla_{x}\rho_{\epsilon}\left(x-y\right)\bigg)\cdot q\left(x,u,\tilde{u}\right)dyds
\\
+\intop_{I}\intop_{\mathbb{R}^d}\psi(x,t)\omega_\epsilon(t-s)\nabla_{y}\rho_{\epsilon}\left(x-y\right)\cdot q\left(y,u,\tilde{u}\right)dyds
\\
=\psi(x,t)\intop_{I}\intop_{\mathbb{R}^d}\omega_\epsilon(t-s)\bigg(\nabla_{x}\rho_{\epsilon}\left(x-y\right)\cdot q\left(x,u,\tilde{u}\right)+\nabla_{y}\rho_{\epsilon}\left(x-y\right)\cdot q\left(y,u,\tilde{u}\right)\bigg)dyds
\\
+\nabla_x\psi(x,t)\cdot\intop_{I}\intop_{\mathbb{R}^d}\omega_\epsilon(t-s)\rho_{\epsilon}\left(x-y\right)q\left(x,u,\tilde{u}\right)dyds
\\
=\psi(x,t)\intop_{I}\intop_{\mathbb{R}^d}\omega_\epsilon(t-s)\nabla_{y}\rho_{\epsilon}(x-y)\cdot\bigg(q(y,u,\tilde{u})-q(x,u,\tilde{u})\bigg)dyds
\\
+\nabla_x\psi(x,t)\cdot\intop_{I}\intop_{\mathbb{R}^d}\omega_\epsilon(t-s)\rho_{\epsilon}\left(x-y\right)q\left(x,u,\tilde{u}\right)dyds,
\end{multline}
where in the last equation of \eqref{eq:sixth integral R} we use $\nabla_{x}\rho_{\epsilon}\left(x-y\right)=-\nabla_{y}\rho_{\epsilon}\left(x-y\right)$.

We substitute (\ref{eq:fourth integral R}), (\ref{eq:fifth integral R}), and (\ref{eq:sixth integral R}) into (\ref{eq:third integral R}) to obtain for $\epsilon \in (0,\infty)$
\begin{equation}
\label{eq:eighth integral R}
\intop_{I}\intop_{\mathbb{R}^d}\bigg\{\partial_t\psi(x,t)I_{1}^{\epsilon}(x,t)+\nabla_x\psi(x,t)\cdot I_{2}^{\epsilon}(x,t)+\psi(x,t)\bigg(I_{3}^{\epsilon}(x,t)+I_{4}^{\epsilon}(x,t)\bigg)\bigg\}dxdt
\geq 0.
\end{equation}
Here
\begin{equation}
I_{1}^{\epsilon}\left(x,t\right):=\intop_{I}\intop_{\mathbb{R}^d}\omega_\epsilon(t-s)\rho_{\epsilon}(x-y)\eta\left(u(x,t),\tilde{u}(y,s)\right)dyds,
\end{equation}
\begin{equation}
I_{2}^{\epsilon}\left(x,t\right):=\intop_{I}\intop_{\mathbb{R}^d}\omega_\epsilon(t-s)\rho_{\epsilon}\left(x-y\right)q\left(x,u(x,t),\tilde{u}(y,s)\right)dyds,
\end{equation}

\begin{equation}
I_{3}^{\epsilon}\left(x,t\right):=\intop_{I}\intop_{\mathbb{R}^d}\omega_\epsilon(t-s)\rho_{\epsilon}\left(x-y\right)\sign\left(u(x,t)-\tilde{u}(y,s)\right)
\bigg(\diver_y f\left(y,u(x,t)\right)-\diver_x f\left(x,\tilde{u}(y,s)\right)\bigg)dyds,
\end{equation}

\begin{equation}
I_{4}^{\epsilon}\left(x,t\right):=\intop_{I}\intop_{\mathbb{R}^d}\omega_\epsilon(t-s)\nabla_{y}\rho_{\epsilon}(x-y)\cdot\biggl(q(y,u(x,t),\tilde{u}(y,s))-q(x,u(x,t),\tilde{u}(y,s))\biggr)dyds.
\end{equation}

\textbf{Step 2}\\
We intend to take the limit as $\epsilon\to 0^+$ in \eqref{eq:eighth integral R} using the Dominated Convergence Theorem to interchange the limit and the integral. To apply the Dominated Convergence Theorem, we need to prove the following four assertions, which tell us that the families of functions $I^\epsilon_j$, $j\in\{1,2,3,4\}$, indexed by $\epsilon$ with variable $(x,t)\in \R^d\times I$, converge almost everywhere and are bounded by locally integrable functions:
\begin{enumerate}
\item For almost every $(x,t)\in\R^d\times I$ we get
\begin{equation}
\label{eq:convergence of I1}
\begin{cases}
\lim_{\epsilon\to 0^+}I_{1}^{\epsilon}\left(x,t\right)=\eta(u(x,t),\tilde{u}(x,t))
\\ 
\sup_{\epsilon\in(0,\infty)}|I_{1}^{\epsilon}\left(x,t\right)|\leq \|\tilde{u}\|_{L^\infty(\R^d\times I)}+|u(x,t)|
\end{cases},
\end{equation}
and the function on the right hand side of the inequality in \eqref{eq:convergence of I1} lies in the space $L^1_{\text{loc}}(\R^d\times I)$.

\item For almost every $(x,t)\in\R^d\times I$ we get
\begin{equation}
\label{eq:convergence of I2}
\begin{cases}
\lim_{\epsilon\to 0^+}I_{2}^{\epsilon}\left(x,t\right)=q(x,u(x,t),\tilde{u}(x,t))
\\ 
\sup_{\epsilon\in(0,\infty)}|I_{2}^{\epsilon}\left(x,t\right)|\leq \|f(x,\cdot)\|_{L^\infty(\image \tilde{u})}+|f(x,u(x,t))|
\end{cases},
\end{equation}
and the function on the right hand side of the inequality in \eqref{eq:convergence of I2} lies in the space $L^1_{\text{loc}}(\R^d\times I)$.

\item Let $P_x(\Supp(\psi))$ be the projection on $\R^d$ of the compact support of $\psi$, and let $U_0$ be any open and bounded set which contains it. Let us denote $D:=\dist(P_x(\Supp(\psi)),\partial U_0)>0$. For almost every $(x,t)\in P_x(\Supp(\psi))\times I$ we have
\begin{equation}
\label{eq:convergence of I3}
\begin{cases}
\lim_{\epsilon\to 0^+}I_{3}^{\epsilon}\left(x,t\right)=\diver_xq\left(x,u\left(x,t\right),\tilde{u}\left(x,t\right)\right)
\\
\sup_{\epsilon\in(0,D)}|I_{3}^{\epsilon}\left(x,t\right)|\leq \|\diver_yf(\cdot,u(x,t))\|_{L^\infty(U_0)}+\|\diver_xf(x,\cdot)\|_{L^\infty(\image \tilde{u})}
\end{cases},
\end{equation}
and the function on the right hand side of the inequality in \eqref{eq:convergence of I3} lies in the space $L^1_{\text{loc}}(\R^d\times I)$.

\item Moreover, let $U_0$ and $D$ as above, and let $A\subset\R$ be any bounded set such that $\mathcal{L}^1\big(A\Delta\left[\image u\cup\image \tilde{u}\right]\big)=0$. Let us denote by $L$ the Lipschitz constant of $f$ on the bounded set $U_0\times A$. Then, for almost every $(x,t)\in P_x(\Supp(\psi))\times I$ we get
\begin{equation}
\label{eq:convergence of I4}
\begin{cases}
\lim_{\epsilon\to 0^+}I_{4}^{\epsilon}\left(x,t\right)=-\diver_xq\left(x,u\left(x,t\right),\tilde{u}\left(x,t\right)\right)
\\
\sup_{\epsilon\in(0,D)}|I_{4}^{\epsilon}\left(x,t\right)|\leq 2L\mathcal{L}^d(B_1(0))\|\nabla\rho\|_{L^\infty(\R^d)}\sqrt{d}
\end{cases}.
\end{equation}
\end{enumerate}

From \eqref{eq:eighth integral R}, \eqref{eq:convergence of I1}, \eqref{eq:convergence of I2}, \eqref{eq:convergence of I3}, and \eqref{eq:convergence of I4}, we obtain by the Dominated Convergence Theorem the localized contraction property \eqref{eq:final term, in statement of theorem}.

The proofs of \eqref{eq:convergence of I1}, \eqref{eq:convergence of I2} and \eqref{eq:convergence of I3} rely mainly upon Lebesgue Differentiation Theorem (Theorem \ref{thm:Lebesgue differentiation theorem}). The proof of \eqref{eq:convergence of I4} is a bit more complicated than the others, and it is established using Lebesgue Differentiation Theorem along with an approximating argument. 

We prove \eqref{eq:convergence of I1}. Note that for $t\in (0,\infty)$, we get for every sufficiently small $\epsilon\in (0,\infty)$ that $\intop_{I}\omega_\epsilon(t-s)ds=1$. Hence, by the definition of $\eta$ and the triangle inequality we get
\begin{multline}
\Bigl|I_{1}^{\epsilon}\left(x,t\right)-\eta\left(u\left(x,t\right),\tilde{u}\left(x,t\right)\right)\Bigr|
\\
=\left|\intop_{I}\intop_{\mathbb{R}^d}\omega_\epsilon(t-s)\rho_{\epsilon}(x-y)\eta\left(u(x,t),\tilde{u}(y,s)\right)dyds-\intop_{I}\intop_{\mathbb{R}^d}\omega_\epsilon(t-s)\rho_{\epsilon}(x-y)\eta\left(u(x,t),\tilde{u}(x,t)\right)dyds\right|
\\
\leq\intop_{I}\intop_{\mathbb{R}^d}\omega_\epsilon(t-s)\rho_{\epsilon}\left(x-y\right)\big|\eta\left(u\left(x,t\right),\tilde{u}\left(y,s\right)\right)-\eta\left(u\left(x,t\right),\tilde{u}\left(x,t\right)\right)\big|dyds
\\
\leq  \|\omega\|_{L^\infty(\R)}\|\rho\|_{L^\infty(\R^d)}\frac{1}{\epsilon}\intop_{t-\epsilon}^{t+\epsilon}\frac{1}{\epsilon^d}\intop_{B_{\epsilon}\left(x\right)}\left|\tilde{u}\left(x,t\right)-\tilde{u}\left(y,s\right)\right|dyds.
\end{multline}

Since $\tilde{u}\in L^{\infty}\left(\mathbb{R}^d\times I\right)$, then $\tilde{u}\in L_{\text{loc}}^{1}\left(\mathbb{R}^d\times I\right)$ and, according to Lebesgue Differentiation Theorem,
we get for almost every $(x,t)\in \mathbb{R}^d\times I$ the limit in \eqref{eq:convergence of I1}. The inequality in (\ref{eq:convergence of I1}) follows from the definition of $I^\epsilon_1$ and properties of $\omega_\epsilon,\rho_\epsilon$: for every $\epsilon\in(0,\infty)$ and for almost every $(x,t)\in\R^d\times I$ we have
\begin{equation}
|I_{1}^{\epsilon}\left(x,t\right)|\leq \ess-sup\limits_{(y,s)\in \R^d\times I}\eta(u(x,t),\tilde{u}(y,s))\leq \|\tilde{u}\|_{L^\infty(\R^d\times I)}+|u(x,t)|.
\end{equation}
Since $u,\tilde{u}\in L^\infty(\R^d\times I)$, the function $(x,t)\mapsto \|\tilde{u}\|_{L^\infty(\R^d\times I)}+|u(x,t)|$ lies in the space $L^1_{\text{loc}}(\R^d\times I)$.
It proves (\ref{eq:convergence of I1}). 

We prove \eqref{eq:convergence of I2}. We get
for almost every $(x,t)\in\R^d\times I$  and every small enough $\epsilon\in (0,\infty)$
\begin{multline}
\big|I_{2}^{\epsilon}\left(x,t\right)-q\left(x,u\left(x,t\right),\tilde{u}\left(x,t\right)\right)\big|
\\
=\left|\intop_{I}\intop_{\mathbb{R}^d}\omega_\epsilon(t-s)\rho_{\epsilon}\left(x-y\right)q\left(x,u(x,t),\tilde{u}(y,s)\right)dyds
-\intop_{I}\intop_{\mathbb{R}^d}\omega_\epsilon(t-s)\rho_{\epsilon}\left(x-y\right)q\left(x,u(x,t),\tilde{u}(x,t)\right)dyds\right|
\\
\leq\intop_{I}\intop_{\mathbb{R}^d}\omega_\epsilon(t-s)\rho_{\epsilon}(x-y)\big|q\left(x,u\left(x,t\right),\tilde{u}\left(y,s\right)\right)-q\left(x,u\left(x,t\right),\tilde{u}\left(x,t\right)\right)\big|dyds
\\
\leq \|\omega\|_{L^\infty(\R)}\|\rho\|_{L^\infty(\R^d)}\frac{1}{\epsilon}\intop_{t-\epsilon}^{t+\epsilon}\frac{1}{\epsilon^d}\intop_{B_{\epsilon}\left(x\right)}\left|q\left(x,u\left(x,t\right),\tilde{u}\left(y,s\right)\right)-q\left(x,u\left(x,t\right),\tilde{u}\left(x,t\right)\right)\right|dyds.
\end{multline}
Since $\tilde{u}\in L^{\infty}\left(\mathbb{R}^d\times I\right)$
and $f(x,\cdot)\in L^\infty_{\text{loc}}(\R)$ for every $x\in \R^d$, we know that the function
\begin{equation}
\left(y,s\right)\longmapsto q\left(x,u\left(x,t\right),\tilde{u}\left(y,s\right)\right)=\sign\left(u\left(x,t\right)-\tilde{u}\left(y,s\right)\right)\big(f\left(x,u\left(x,t\right)\right)-f\left(x,\tilde{u}\left(y,s\right)\right)\big)
\end{equation}
lies in $L_{\text{loc}}^{1}\left(\mathbb{R}^d\times I\right)$. Therefore, by Lebesgue Differentiation Theorem, almost every
$\left(x,t\right)\in\mathbb{R}^d\times I$ is a Lebesgue point of this
function, so we get the limit in \eqref{eq:convergence of I2}. The inequality in (\ref{eq:convergence of I2}) follows from the definition of $I^\epsilon_2$: for every $\epsilon\in(0,\infty)$ and almost every $(x,t)\in \R^d\times I$ we get
\begin{multline}
|I_{2}^{\epsilon}\left(x,t\right)|\leq \ess-sup_{(y,s)\in \R^d\times I}|q(x,u(x,t),\tilde{u}(y,s))|\leq \ess-sup_{(y,s)\in \R^d\times I}|f(x,u(x,t))-f(x,\tilde{u}(y,s))|
\\
\leq |f(x,u(x,t))|+\|f(x,\cdot)\|_{L^\infty(\image \tilde{u})}.
\end{multline}
Since $u,\tilde{u}\in L^\infty(\R^d\times I)$ and $f\in L^\infty_{\text{loc}}(\R^d\times\R,\R^d)$ (because $f$ is continuous), the function $(x,t)\mapsto |f(x,u(x,t))|+\|f(x,\cdot)\|_{L^\infty(\image \tilde{u})}$ lies in the space $L^1_{\text{loc}}(\R^d\times I)$.
It proves (\ref{eq:convergence of I2}). 

We prove now \eqref{eq:convergence of I3}. For almost every $(x,t)\in\R^d\times I$ and sufficiently small $\epsilon\in (0,\infty)$ we have    
\begin{multline}
\label{eq:some estimate for I3}
\left|I_{3}^{\epsilon}\left(x,t\right)-\diver_xq(x,u(x,t),\tilde{u}(x,t))\right|
\\
=\Bigg|\intop_{I}\intop_{\mathbb{R}^d}\omega_\epsilon(t-s)\rho_{\epsilon}\left(x-y\right)\sign\left(u(x,t)-\tilde{u}(y,s)\right)
\bigg(\diver_y f\left(y,u(x,t)\right)-\diver_x f\left(x,\tilde{u}(y,s)\right)\bigg)dyds
\\
-\intop_{I}\intop_{\mathbb{R}^d}\omega_\epsilon(t-s)\rho_{\epsilon}\left(x-y\right)\diver_x q(x,u(x,t),\tilde{u}(x,t))dyds\Bigg|
\\
\leq\intop_{I}\intop_{\mathbb{R}^d}\omega_\epsilon(t-s)\rho_{\epsilon}\left(x-y\right)\bigg|\sign\left(u(x,t)-\tilde{u}(y,s)\right)
\bigg(\diver_y f\left(y,u(x,t)\right)-\diver_x f\left(x,\tilde{u}(y,s)\right)\bigg)
\\
-\diver_x q(x,u(x,t),\tilde{u}(x,t))\bigg|dyds
\\
\leq \|\omega\|_{L^\infty(\R)}\|\rho\|_{L^\infty(\R^d)}\frac{1}{\epsilon}\intop_{t-\epsilon}^{t+\epsilon}\frac{1}{\epsilon^d}\intop_{B_{\epsilon}\left(x\right)}\bigg|\sign\left(u(x,t)-\tilde{u}(y,s)\right)
\bigg(\diver_y f\left(y,u(x,t)\right)-\diver_x f\left(x,\tilde{u}(y,s)\right)\bigg)
\\
-\diver_x q(x,u(x,t),\tilde{u}(x,t))\bigg|dyds.
\end{multline}

Since $\tilde{u}\in L^\infty(\R^d\times I)$ and $f$ is locally Lipschitz, we get for almost every $(x,t)\in \R^d\times I$ that the function
\begin{equation}
\label{eq:a function with divergence}
(y,s)\longmapsto \sign(u(x,t)-\tilde{u}(y,s))\bigg(\diver_y f(y,u(x,t))-\diver_x f(x,\tilde{u}(y,s))\bigg)
\end{equation}
lies in the space $L^1_{\text{loc}}(\R^d\times I)$. Indeed, note that for $k:=u(x,t)$, the function $y\mapsto f(y,k)$ is locally Lipschitz in $\R^d$. Therefore, by Rademacher's Theorem, we get that  
$\diver_yf(\cdot,k)\in L^\infty_{\text{loc}}(\R^d)$, so the function $(y,s)\mapsto \diver_yf(y,u(x,t))$ lies in $L^1_{\text{loc}}(\R^d\times I)$. By item 2a in Assumption \ref{ass:assumption on the flux}, for almost every $x\in\R^d$, we get that the function $k\mapsto \diver_xf(x,k)$ is continuous on $\R$, so it is locally bounded on $\R$. Since $\tilde{u}\in L^\infty(\R^d\times I)$, we have that the function $(y,s)\mapsto \diver_x f(x,\tilde{u}(y,s))$ is bounded in $\R^d\times I$. Therefore, we get that the function in \eqref{eq:a function with divergence} lies in $L^1_{\text{loc}}(\R^d\times I)$. Thus, almost every $(x,t)\in\R^d\times I$ is a Lebesgue point of this function according to Lebesgue Differentiation Theorem. Therefore, for almost every $(x,t)\in\R^d\times I$, we obtain by \eqref{eq:some estimate for I3}, the limit in \eqref{eq:convergence of I3}. The inequality in \eqref{eq:convergence of I3} follows from the definition of $I_3^\epsilon$:
for almost every $(x,t)\in P_x(\Supp(\psi))\times I$ and every $\epsilon\in (0,D)$ we have $B_\epsilon(x)\subset U_0$ and
\begin{multline}
\label{eq:estimate for I3(x,t)}
|I_{3}^{\epsilon}\left(x,t\right)|\leq \intop_{I}\intop_{\mathbb{R}^d}\omega_\epsilon(t-s)\rho_{\epsilon}\left(x-y\right)\bigg|\diver_y f\left(y,u(x,t)\right)-\diver_x f\left(x,\tilde{u}(y,s)\right)\bigg|dyds
\\
\leq \ess-sup_{(y,s)\in U_0\times I}\big|\diver_y f\left(y,u(x,t)\right)-\diver_x f\left(x,\tilde{u}(y,s)\right)\big|
\\
\leq \|\diver_yf(\cdot,u(x,t))\|_{L^\infty(U_0)}+\|\diver_xf(x,\cdot)\|_{L^\infty(\image \tilde{u})}.
\end{multline}
The function on the right hand side of \eqref{eq:estimate for I3(x,t)} is locally bounded in $\R^d\times I$ since $u,\tilde{u}\in  L^\infty(\R^d\times I)$ and $\diver_x f\in L^\infty_{\text{loc}}(\R^d\times \R)$.
It completes the proof of  \eqref{eq:convergence of I3}. 

Now we prove \eqref{eq:convergence of I4}. First note that by the definition of the inner product we get
\begin{equation}
\label{eq: sum of partial derivatives of mollifier kernel times q   R}
I_{4}^{\epsilon}\left(x,t\right)
=-\sum_{i=1}^d\intop_{I}\intop_{\mathbb{R}^d}\omega_\epsilon(t-s)\partial_{y_i}\rho_{\epsilon}\left(x-y\right)\, \bigg(q_i\left(x,u\left(x,t\right),\tilde{u}\left(y,s\right)\right)-q_i\left(y,u\left(x,t\right),\tilde{u}\left(y,s\right)\right)\bigg)dyds,
\end{equation}
where we denote $q=(q_1,...,q_d)$. Therefore, for proving the limit in \eqref{eq:convergence of I4}, according to \eqref{eq: sum of partial derivatives of mollifier kernel times q   R}, it is enough to show that for every natural number $1\leq i\leq d$ and almost every $(x,t)\in \R^d\times I$ 
\begin{multline}
\label{eq: approximation of the i-th coordinate of I42 3}
\lim_{\epsilon\to 0^+}\intop_{I}\intop_{\mathbb{R}^d}\omega_\epsilon(t-s)\partial_{y_i}\rho_{\epsilon}\left(x-y\right)\, \bigg(q_i\left(x,u\left(x,t\right),\tilde{u}\left(y,s\right)\right)-q_i\left(y,u\left(x,t\right),\tilde{u}\left(y,s\right)\right)\bigg)dyds
\\
=(q_i)_{x_i}\left(x,u\left(x,t\right),\tilde{u}\left(x,t\right)\right)=e_i\cdot \nabla_xq_i\left(x,u\left(x,t\right),\tilde{u}\left(x,t\right)\right),
\end{multline} 
where $e_i$ is the standard unit vector with $1$ in the $i$-th coordinate and zero in the other coordinates. Here $(q_i)_{x_i}\left(x,u\left(x,t\right),\tilde{u}\left(x,t\right)\right)=\partial_{x_i}q_i\left(x,k_1,k_2\right)|_{k_1=u(x,t),k_2=\tilde{u}(x,t)}$.

Let us fix $1\leq i\leq d,i\in \N$. We now prove that, for almost every $(x,t)\in \R^d\times I$ the following formula holds:
\begin{equation}
\label{eq:helpformula2}
q_i\left(x,u\left(x,t\right),\tilde{u}\left(y,s\right)\right)-q_i\left(y,u\left(x,t\right),\tilde{u}\left(y,s\right)\right)
=\nabla_{x}q_i(x,u(x,t),\tilde{u}(y,s))\cdot\left(x-y\right)+\Theta(x,t,y,s)
\end{equation}
for almost every $(y,s)\in \R^d\times I$. Here $\Theta(x,t,y,s)$ is a function with the following property: for every $\xi\in (0,\infty)$ there exists $\epsilon\in(0,\infty)$ such that for almost every $(y,s)\in B_\epsilon(x)\times I$ we have $|\Theta(x,t,y,s)|\leq\xi|x-y|$.

Let us denote $f=(f_1,...,f_d)$. Let $(x_0,t_0)\in \R^d\times I$ be any point in the domain of $u$ such that the derivative of the $i$-th function coordinate of the flux, $f_i$, at the point $x_0$, $\nabla_xf_i(x_0,k)$, is a continuous function in $k\in\R$; by item 2a of Assumption \ref{ass:assumption on the flux} almost every $x\in\R^d$ has this property. Let $(y_0,s_0)\in \R^d\times I$ be a point in the domain of $\tilde{u}$. Let us denote $k_{1}:=u\left(x_0,t_0\right)$ and $k_{2}:=\tilde{u}\left(y_0,s_0\right)$. Let us define
\begin{equation}
g(y):=q_i\left(y,k_{1},k_{2}\right)=\sign\left(k_{1}-k_{2}\right)\left(f_i\left(y,k_{1}\right)-f_i\left(y,k_{2}\right)\right).
\end{equation}
By the choice of $x_0$, the function $g$ is differentiable at $x_0$ and  
\begin{equation}
\label{eq:helpformula1}
g(x_0)-g(y)=\nabla g(x_0)\cdot (x_0-y)+o(x_0-y),\quad \forall y\in\R^d.
\end{equation}
Therefore, we get \eqref{eq:helpformula1} for $y=y_0$. It proves \eqref{eq:helpformula2}. Note that 

\begin{multline}
\frac{|\Theta(x,t,y,s)|}{|x-y|}
=\frac{\left|q_i\left(x,u\left(x,t\right),\tilde{u}\left(y,s\right)\right)-q_i\left(y,u\left(x,t\right),\tilde{u}\left(y,s\right)\right)
-\nabla_{x}q_i(x,u(x,t),\tilde{u}(y,s))\cdot\left(x-y\right)\right|}{|x-y|}
\\
\leq \frac{\left|f_i(x,u(x,t))-f_i(y,u(x,t))
-\nabla_{x}f_i(x,u(x,t))\cdot(x-y)\right|}{|x-y|}
\\
+\frac{\left|f_i(x,\tilde{u}(y,s))-f_i(y,\tilde{u}(y,s))
-\nabla_{x}f_i(x,\tilde{u}(y,s))\cdot(x-y)\right|}{|x-y|}
\\
\leq 2\sup_{k\in K}\frac{\left|f_i(x,k)-f_i(y,k)
-\nabla_{x}f_i(x,k)\cdot(x-y)\right|}{|x-y|},
\end{multline}
where $K\subset\R$ is any compact set for which $\mathcal{L}^1\left(K\Delta\left[\image u\cup \image\tilde{u}\right]\right)=0$. By item 2b of Assumption \ref{ass:assumption on the flux} we have for arbitrary positive number $\xi$ a number $\epsilon\in (0,\infty)$ such that $|\Theta(x,t,y,s)|\leq\xi|x-y|$ for almost every $(y,s)\in B_\epsilon(x)\times I$.

We get by \eqref{eq:helpformula2}
\begin{multline}
\label{eq: approximation of the i-th coordinate of I42}
\intop_{I}\intop_{\mathbb{R}^d}\omega_\epsilon(t-s)\partial_{y_i}\rho_{\epsilon}\left(x-y\right)\, \bigg(q_i\left(x,u\left(x,t\right),\tilde{u}\left(y,s\right)\right)-q_i\left(y,u\left(x,t\right),\tilde{u}\left(y,s\right)\right)\bigg)dyds
\\
=\intop_{I}\intop_{\mathbb{R}^d}\omega_\epsilon(t-s)\partial_{y_i}\rho_{\epsilon}\left(x-y\right)\left(x-y\right)\cdot \nabla_xq_i\left(x,u\left(x,t\right),\tilde{u}\left(y,s\right)\right)dyds+o_\epsilon\left(1\right),
\end{multline}
where
\begin{equation}
\label{eq:litle o of an expression}
o_\epsilon\left(1\right):=\intop_{I}\intop_{\mathbb{R}^d}\omega_\epsilon(t-s)\partial_{y_i}\rho_{\epsilon}\left(x-y\right)\Theta(x,t,y,s)dyds\quad \text{and}\quad \lim_{\epsilon\to 0^+}o_\epsilon\left(1\right)=0.
\end{equation}
In order to prove the limit in \eqref{eq:litle o of an expression}, note that for $y\in B_{\epsilon}\left(x\right)$, we get
\begin{multline}
\label{eq:bound for derivative of rho epsilon}
\left|\partial_{y_i}\rho_{\epsilon}(x-y)\right||x-y|=\left|\frac{1}{\epsilon^d}\nabla \rho\left(\frac{x-y}{\epsilon}\right)\cdot\left(-\frac{1}{\epsilon}e_i\right)\right||x-y|
\leq\left|\frac{1}{\epsilon^d}\nabla \rho\left(\frac{x-y}{\epsilon}\right)\right|
\leq \frac{1}{\epsilon^d}\|\nabla \rho \|_{L^\infty(\R^d)},
\end{multline}
and for arbitrarily small $\xi\in (0,\infty)$, there exists $\epsilon\in (0,\infty)$ such that for almost every $(y,s)\in B_\epsilon(x)\times I$ we have $|\Theta(x,t,y,s)|\leq \xi |x-y|$. Therefore, $|o_\epsilon\left(1\right)|\leq \|\nabla \rho \|_{L^\infty(\R^d)}\mathcal{L}^d(B_1(0))\xi$.

In addition, note that by Fubini's Theorem and integration by parts we have for every sufficiently small $\epsilon\in (0,\infty)$

\begin{equation}
\label{eq:identity for the unit vector}
\intop_{I}\intop_{\mathbb{R}^d}\omega_\epsilon(t-s)\partial_{y_i}\rho_{\epsilon}\left(x-y\right)\left(x-y\right)dyds=e_i.
\end{equation}

More precisely, note that on the left-hand side of \eqref{eq:identity for the unit vector}, we have a vector due to the expression $x-y$. Examine the $j$-th coordinate of this vector, meaning that $\intop_{I}\intop_{\mathbb{R}^d}\omega_\epsilon(t-s)\partial_{y_i}\rho_{\epsilon}\left(x-y\right)\left(x_j-y_j\right)dyds$. Utilize Fubini's Theorem and integration by parts to transfer the derivative $\partial_{y_i}$ to $x_j-y_j$. In the case where $j\neq i$, we obtain zero; in the case where $j=i$, we obtain $1$.

Hence, using \eqref{eq:bound for derivative of rho epsilon} and \eqref{eq:identity for the unit vector}, we obtain
\begin{multline}
\label{eq: approximation of the i-th coordinate of I42 2}
\bigg|\intop_{I}\intop_{\mathbb{R}^d}\omega_\epsilon(t-s)\partial_{y_i}\rho_{\epsilon}\left(x-y\right)\left(x-y\right)\cdot \nabla_xq_i\left(x,u\left(x,t\right),\tilde{u}\left(y,s\right)\right)dyds-e_i\cdot \nabla_xq_i\left(x,u\left(x,t\right),\tilde{u}\left(x,t\right)\right)\bigg|
\\
=\bigg|\intop_{I}\intop_{\mathbb{R}^d}\omega_\epsilon(t-s)\partial_{y_i}\rho_{\epsilon}\left(x-y\right)\left(x-y\right)\cdot \nabla_xq_i\left(x,u\left(x,t\right),\tilde{u}\left(y,s\right)\right)dyds
\\
-\left(\intop_{I}\intop_{\mathbb{R}^d}\omega_\epsilon(t-s)\partial_{y_i}\rho_{\epsilon}\left(x-y\right)\left(x-y\right)dyds\right)\cdot \nabla_xq_i\left(x,u\left(x,t\right),\tilde{u}\left(x,t\right)\right)\bigg|
\\
\leq\intop_{I}\intop_{\mathbb{R}^d}\omega_\epsilon(t-s)\left|\partial_{y_i}\rho_{\epsilon}(x-y)\right||x-y|\bigg|\nabla_xq_i\left(x,u\left(x,t\right),\tilde{u}\left(y,s\right)\right)-\nabla_xq_i\left(x,u\left(x,t\right),\tilde{u}\left(x,t\right)\right)\bigg|dyds
\\
\leq \|\omega\|_{L^\infty(\R)}\|\nabla \rho \|_{L^\infty(\R^d)}\frac{1}{\epsilon}\intop_{t-\epsilon}^{t+\epsilon}\frac{1}{\epsilon^d}\intop_{B_{\epsilon}\left(x\right)}\bigg|\nabla_xq_i\left(x,u\left(x,t\right),\tilde{u}\left(y,s\right)\right)-\nabla_xq_i\left(x,u\left(x,t\right),\tilde{u}\left(x,t\right)\right)\bigg|dyds.
\end{multline}
By item 2a of Assumption \ref{ass:assumption on the flux} on the flux $f$ and $\tilde{u}\in L^\infty(\R^d\times I)$, we get that
\begin{equation}
(y,s)\longmapsto \nabla_xq_i\left(x,u\left(x,t\right),\tilde{u}\left(y,s\right)\right)
\end{equation}
lies in the space $L_{\text{loc}}^1(\R^d\times I)$. Therefore, by Lebesgue Differentiation Theorem, \eqref{eq: approximation of the i-th coordinate of I42} and \eqref{eq: approximation of the i-th coordinate of I42 2}, we obtain \eqref{eq: approximation of the i-th coordinate of I42 3}, from which we get the limit in $\eqref{eq:convergence of I4}$. For the inequality in $\eqref{eq:convergence of I4}$, note that, for $k_1:=u(x,t)$ and
$k_{2}:=\tilde{u}\left(y,s\right)$ we get
\begin{multline}
\label{eq:expression for difference between q(x) and q(y)}
q(y,k_{1},k_{2})-q(x,k_{1},k_{2})
=\sign(k_{1}-k_{2})(f(y,k_{1})-f(y,k_{2}))-\sign(k_{1}-k_{2})(f(x,k_{1})-f(x,k_{2}))
\\
=\sign(k_{1}-k_{2})\bigg(f(y,k_{1})-f(x,k_1)+f(x,k_2)-f(y,k_2)\bigg).
\end{multline}

Let $U_0,A,D$ and $L$ as in the formulation above \eqref{eq:convergence of I4}. Notice that for almost every $(x,t)\in P_x(\Supp(\psi))\times I$ and for every $\epsilon\in(0,D)$, we get by \eqref{eq:expression for difference between q(x) and q(y)}

\begin{multline}
\label{eq:estimate for Iepsilon4}
|I_{4}^{\epsilon}\left(x,t\right)|
\leq\intop_{I}\intop_{\mathbb{R}^d}\omega_\epsilon(t-s)|\nabla_{y}\rho_{\epsilon}(x-y)|\Bigl| q(y,u(x,t),\tilde{u}(y,s))-q(x,u(x,t),\tilde{u}(y,s))\Bigr|dyds
\\
\leq\intop_{I}\intop_{B_\epsilon(x)}\omega_\epsilon(t-s)|\nabla_{y}\rho_{\epsilon}(x-y)|\Bigl|f(x,u(x,t))-f(y,u(x,t))\Bigr|dyds
\\
+\intop_{I}\intop_{B_\epsilon(x)}\omega_\epsilon(t-s)|\nabla_{y}\rho_{\epsilon}(x-y)|\Bigl|f(x,\tilde{u}(y,s))-f(y,\tilde{u}(y,s))\Bigr|dyds
\\
\leq 2L\intop_{I}\intop_{B_\epsilon(x)}\omega_\epsilon(t-s)|\nabla_{y}\rho_{\epsilon}(x-y)||x-y|dyds
\leq 2L\mathcal{L}^d(B_1(0))\|\nabla\rho\|_{L^\infty(\R^d)}\sqrt{d},
\end{multline}
where in the last inequality of \eqref{eq:estimate for Iepsilon4} we use: for every $y\in B_\epsilon(x)$ 
\begin{multline}
\label{eq:bound for derivative of rho epsilon1}
\left|\nabla_{y}\rho_{\epsilon}(x-y)\right||x-y|=\left|\frac{1}{\epsilon^d}\nabla \rho\left(\frac{x-y}{\epsilon}\right)\cdot\left(-\frac{1}{\epsilon}I_{d\times d}\right)\right||x-y|
\\
\leq\left|\frac{1}{\epsilon^d}\nabla \rho\left(\frac{x-y}{\epsilon}\right)\right|\sqrt{d}
\leq \frac{1}{\epsilon^d}\|\nabla \rho \|_{L^\infty(\R^d)}\sqrt{d},
\end{multline}
where $I_{d\times d}$ is the identity matrix of size $d\times d$.
It completes the proof of \eqref{eq:convergence of I4}.
\end{proof} 

\begin{thm}
\label{thm:localized $L^1$-contraction property}
(Local $L^1$-contraction property)
Let $u$,$\tilde{u}$ be two entropy solutions of $\eqref{eq: The Cauchy problem for the scalar conservation law}$, and $f$ is a flux as in Assumption \ref{ass:assumption on the flux}. Let us define
\begin{equation}
\label{eq:definition of M}
M:=\max\left\{\nnorm[u]_{L^\infty(\R^d\times I)},\nnorm[\tilde{u}]_{L^\infty(\R^d\times I)}\right\}<\infty,\quad I=(0,\infty),
\end{equation}
and for each $R\in (0,\infty)$ we define 
\begin{equation}
\label{eq:definition of N in the proof of uniqueness}
N:=N_M(R):=\sup\Set{\frac{|f(x,k)-f(x,k')|}{|k-k'|}}[x\in B_R(0),\,{k,k'\in [-M,M]},\,k\neq k']<\infty.
\end{equation}
The finiteness in \eqref{eq:definition of N in the proof of uniqueness} is due to the assumption that the flux $f=f(x,k)$ is locally Lipschitz in $\R^d\times \R$.
Then, there exists a set $\mathcal{N}\subset (0,N^{-1}R)$ such that $\mathcal{L}^1(\mathcal{N})=0$ and 
\begin{equation}
\intop_{B_{R-\tau N}(0)}\left|u(x,\tau)-\tilde{u}(x,\tau)\right|dx\leq\intop_{B_{R-\rho N}(0)}\left|u(x,\rho)-\tilde{u}(x,\rho)\right|dx,
\end{equation}
for every $\rho,\tau\in (0,N^{-1}R)\setminus \mathcal{N}$ with $\rho\leq \tau$. In other words, the function $t\mapsto \left\|u\left(\cdot,t\right)-\tilde{u}\left(\cdot,t\right)\right\|_{L^1(B_{R-tN}(0))}$ is non-increasing on $(0,N^{-1}R)$ outside a set of measure zero.
\end{thm}
The proof we give here for  Theorem \ref{thm:localized $L^1$-contraction property} was originally given by Kruzkov (see \cite{Kruzkov}). However, we provide the proof here for the sake of completeness and with additional details.

\begin{proof}
We divide the proof into three parts. In the first part, we introduce the functions $\alpha_h$ and $\chi_\epsilon$ and establish their properties. In the second part, we construct a test function $\psi$ for the localized inequality \eqref{eq:final term, in statement of theorem} using the functions $\alpha_h$ and $\chi_\epsilon$, and derive estimates for $\partial_t\psi$ and $\nabla_x\psi$. In the third part, we use the localized inequality and the choice of $\psi$ to establish the local $L^1$-contraction property.
\\
\textbf{Part 1}
\\
For the numbers $N,R$, we define a set
\begin{equation}
\mathcal{K}:=\Set{(x,t)\in \R^d\times I}[t\in {(0,N^{-1}R)},\,\,x\in B_{R-tN}(0)].
\end{equation}
The set $\mathcal{K}$ is an open cone with base in $B_R(0) \times \{0\} \subset \mathbb{R}^d \times \mathbb{R}$ and vertex at the point $(0, N^{-1}R) \in \mathbb{R}^d \times \mathbb{R}$.
 
Let us define for every $h\in (0,\infty)$
\begin{equation}
\label{eq:definition of alpha h}
\alpha_h(\sigma):=\intop_{-\infty}^\sigma\omega_h(s)ds,\quad \sigma\in \R,
\end{equation}
where $\omega_h$ is the one-dimensional mollifier kernel as defined in \eqref{eq: standard mollifier kernel}. Recall the three properties of $\omega_h$: $\omega_h\geq 0$, $\intop_{-\infty}^\infty\omega_h(s)ds=1$, and $\Supp(\omega_h)=[-h,h]$. 
Note that the function $\alpha_h$ has the properties: $0\leq \alpha_h\leq 1$, and it is monotonically non-decreasing on $\R$. For every $\epsilon\in (0,\infty)$, let us define the function
\begin{equation}
\label{eq:definition of chi epsilon}
\chi_\epsilon(x,t):=1-\alpha_\epsilon\left(|x|-[R-tN]+\epsilon\right),\quad (x,t)\in \R^d\times I.
\end{equation}
Note that
\begin{equation}
\label{eq:bound on chi epsilon}
\forall\epsilon\in (0,\infty),\\\forall(x,t)\in \R^d\times I\quad \Longrightarrow\quad  0\leq \chi_\epsilon(x,t)\leq 1.
\end{equation}
In addition, we have
\begin{equation}
\label{eq:chi epsilon vanishes outside the cone K}
\forall\epsilon\in(0,\infty),\forall (x,t)\in\left(\R^d\times I\right)\setminus \mathcal{K}\quad \Longrightarrow\quad \chi_\epsilon(x,t)=0.
\end{equation}
Indeed, if $(x,t)\in\left(\R^d\times I\right)\setminus \mathcal{K}$, then $t\geq N^{-1}R$ or $|x|\geq  R-tN$. If $t\geq N^{-1}R$, then $|x|-[R-tN]+\epsilon\geq |x|+\epsilon$; if $|x|\geq R-tN$, then $|x|-[R-tN]+\epsilon\geq \epsilon$. Hence, in both cases $1\geq \alpha_\epsilon\left(|x|-[R-tN]+\epsilon\right)\geq \alpha_\epsilon(\epsilon)=1$. Therefore, $\chi_\epsilon(x,t)=1-\alpha_\epsilon\left(|x|-[R-tN]+\epsilon\right)=0$. This proves \eqref{eq:chi epsilon vanishes outside the cone K}.

The next property of $\chi_\epsilon$ we need is the following: for every $(x,t)\in \R^d\times I$
\begin{equation}
\label{eq:limit of chi epsilon converges to the indicator of the cone}
\lim_{\epsilon\to 0^+}\chi_\epsilon(x,t)=\chi_{\mathcal{K}}(x,t),
\end{equation}
where $\chi_{\mathcal{K}}$ is the characteristic function of $\mathcal{K}$. In case $(x,t)\in \left(\R^d\times I\right)\setminus \mathcal{K}$, we get  the limit in \eqref{eq:limit of chi epsilon converges to the indicator of the cone} from \eqref{eq:chi epsilon vanishes outside the cone K}. For $(x,t)\in \mathcal{K}$ we denote $\xi(x,t):=|x|-[R-tN]<0$. For every $\epsilon$ such that $\xi(x,t)+\epsilon<-\epsilon$ we get
\begin{equation}
0\leq \alpha_\epsilon\left(|x|-[R-tN]+\epsilon\right)=\alpha_\epsilon\left(\xi(x,t)+\epsilon\right)\leq \alpha_\epsilon(-\epsilon)=\intop_{-\infty}^{-\epsilon}\omega_\epsilon(s)ds=0.
\end{equation}
Hence, for every $(x,t)\in \mathcal{K}$, we have the limit
\begin{equation}
\lim_{\epsilon\to 0^+}\chi_\epsilon(x,t)=1-\lim_{\epsilon\to 0^+}\alpha_\epsilon\left(|x|-[R-tN]+\epsilon\right)=1.
\end{equation}
Therefore, we get \eqref{eq:limit of chi epsilon converges to the indicator of the cone}.
\\
\textbf{Part 2}
\\
Notice that since $u,\tilde{u}\in L^\infty(\R^d\times I)$, the function 
\begin{equation}
\label{eq:Lebesgue points of integral bounded function}
t\longmapsto \int_{B_{R-tN}(0)}|u(x,t)-\tilde{u}(x,t)|dx
\end{equation}
is integrable on $(0,N^{-1}R)$. Indeed, we have the estimate:
\begin{multline}
\int_{(0,N^{-1}R)}\left[\int_{B_{R-tN}(0)}|u(x,t)-\tilde{u}(x,t)|dx\right]dt\leq \int_{(0,N^{-1}R)}\left[\int_{B_{R}(0)}|u(x,t)-\tilde{u}(x,t)|dx\right]dt
\\
\leq \left(\nnorm[u]_{L^\infty(\R^d\times I)}+\nnorm[\tilde{u}]_{L^\infty(\R^d\times I)}\right)N^{-1}R\,\mathcal{L}^d(B_R(0))<\infty.
\end{multline}
Hence almost every point is a Lebesgue point of this function. Let $0<\rho<\tau<N^{-1}R$ be Lebesgue points of this function. Using \eqref{eq:definition of alpha h} and \eqref{eq:definition of chi epsilon}, we define 
\begin{equation}
\psi(x,t):=\left(\alpha_h(t-\rho)-\alpha_h(t-\tau)\right)\chi_\epsilon(x,t).
\end{equation} 
We will use the function $\psi$ as a test function in the localized inequality \eqref{eq:final term, in statement of theorem}. Let us prove that $\psi\geq 0$ and $\psi\in \operatorname{Lip}_{c}\left(\R^d\times I\right)$. Since $\alpha_h$ is monotonically non-decreasing, and \eqref{eq:bound on chi epsilon}, we get that $\psi$ is a product of non-negative functions and hence non-negative.
By \eqref{eq:chi epsilon vanishes outside the cone K}, we have $\psi(x,t)=0$ whenever $|x|>R$ or $t>N^{-1}R$ because $\chi_\epsilon(x,t)=0$. Note that if $0<h<\rho$ and $t<\rho-h$, then 
\begin{equation}
0\leq \alpha_h(t-\tau)\leq \alpha_h(t-\rho)=\int_{-\infty}^{t-\rho}\omega_h(s)ds=0.
\end{equation}  
Therefore, $\psi(x,t)=0$ for $0<t<\rho-h$. Hence, $\Supp(\psi)\subset \overline{B}_R(0)\times [\rho-h,N^{-1}R]$, so it is a compact set in $\R^d\times I$ provided $0<h<\rho$. 

Note that $\psi:\R^d\times I\to [0,\infty)$ is continuous function. Let us compute the partial derivatives of $\psi$. By the Fundamental Theorem of Calculus we have 
\begin{equation}
\label{eq:time derivation of alpha h}
\partial_t\left(\alpha_h(t-\rho)-\alpha_h(t-\tau)\right)=\omega_h(t-\rho)-\omega_h(t-\tau)
\end{equation}
and
\begin{equation}
\label{eq:time derivation of chi epsilon}
\partial_t\chi_\epsilon(x,t)=-\omega_\epsilon\left(|x|-[R-tN]+\epsilon\right)N.
\end{equation}
Therefore, by \eqref{eq:time derivation of alpha h} and \eqref{eq:time derivation of chi epsilon}, we obtain
\begin{multline}
\label{eq:time derivation of psi}
\partial_t\psi(x,t)=\partial_t\left(\alpha_h(t-\rho)-\alpha_h(t-\tau)\right)\chi_\epsilon(x,t)+\left(\alpha_h(t-\rho)-\alpha_h(t-\tau)\right)\partial_t\chi_\epsilon(x,t)
\\
=\left(\omega_h(t-\rho)-\omega_h(t-\tau)\right)\chi_\epsilon(x,t)+\left(\alpha_h(t-\rho)-\alpha_h(t-\tau)\right)\left(-\omega_\epsilon\left(|x|-[R-tN]+\epsilon\right)N\right).
\end{multline}
From \eqref{eq:time derivation of psi}, we get the estimate 
\begin{equation}
\label{eq:estimate for time derivation of psi}
|\partial_t\psi(x,t)|\leq 2\|\omega_h\|_{L^\infty(\R)}
+2N\|\omega_\epsilon\|_{L^\infty(\R)},\quad \forall (x,t)\in \R^d\times I.
\end{equation}
For every $x\neq 0$, we get by the Fundamental Theorem of Calculus
\begin{equation}
\label{eq:spatial derivation of chi}
\nabla_x\chi_\epsilon(x,t)=-\omega_\epsilon\left(|x|-[R-tN]+\epsilon\right)\frac{x}{|x|}.
\end{equation}
Therefore, using \eqref{eq:spatial derivation of chi}, we get
\begin{multline}
\label{eq:spatial derivation of psi}
\nabla_x\psi(x,t)=\left(\alpha_h(t-\rho)-\alpha_h(t-\tau)\right)\nabla_x\chi_\epsilon(x,t)
\\
=\left(\alpha_h(t-\rho)-\alpha_h(t-\tau)\right)\left(-\omega_\epsilon\left(|x|-[R-tN]+\epsilon\right)\frac{x}{|x|}\right).
\end{multline}
Therefore, by \eqref{eq:spatial derivation of psi}, we have the estimate 
\begin{equation}
\label{eq:estimate for spatial derivation of psi}
|\nabla_x\psi(x,t)|\leq 2\|\omega_\epsilon\|_{L^\infty(\R)},\quad \forall(x,t)\in \left(\R^d\setminus\{0\}\right)\times I.
\end{equation}
Since $\psi$ is continuous in $\R^d\times I$, and we have the bounds \eqref{eq:estimate for time derivation of psi}, \eqref{eq:estimate for spatial derivation of psi}, we conclude that $\psi$ is Lipschitz in $\R^d\times I$. Therefore, $\psi$ is a legal test function for the localized inequality \eqref{eq:final term, in statement of theorem}.
\\
\textbf{Part 3}
\\
Let $(\eta,q)$ be as in \eqref{eq:definition of eta and the entropy flux in the contraction property}. By Lemma \ref{lem:Kato's inequality}, we get 
\begin{equation}
\label{eq:subsistution into Kato inequality}
\iint_{\mathcal{K}}\bigg[\partial_t\psi(x,t)\eta\left(u\left(x,t\right),\tilde{u}\left(x,t\right)\right)+\nabla_x\psi(x,t)\cdot q\left(x,u\left(x,t\right),\tilde{u}\left(x,t\right)\right)\bigg]dxdt\geq 0.
\end{equation} 
From \eqref{eq:time derivation of psi} and \eqref{eq:spatial derivation of psi} we have
\begin{multline}
\label{eq:the entropy integrand}
\partial_t\psi(x,t)\eta\left(u\left(x,t\right),\tilde{u}\left(x,t\right)\right)+\nabla_x\psi(x,t)\cdot q\left(x,u\left(x,t\right),\tilde{u}\left(x,t\right)\right)
\\
=\Big[\left(\omega_h(t-\rho)-\omega_h(t-\tau)\right)\chi_\epsilon(x,t)+\left(\alpha_h(t-\rho)-\alpha_h(t-\tau)\right)\left(-\omega_\epsilon\left(|x|-[R-tN]+\epsilon\right)N\right)\Big]\eta\left(u\left(x,t\right),\tilde{u}\left(x,t\right)\right)
\\
+\Big[\left(\alpha_h(t-\rho)-\alpha_h(t-\tau)\right)\left(-\omega_\epsilon\left(|x|-[R-tN]+\epsilon\right)\frac{x}{|x|}\right)\Big]\cdot q\left(x,u\left(x,t\right),\tilde{u}\left(x,t\right)\right)
\\
=\left(\omega_h(t-\rho)-\omega_h(t-\tau)\right)\chi_\epsilon(x,t)\eta\left(u\left(x,t\right),\tilde{u}\left(x,t\right)\right)
\\
-\omega_\epsilon\left(|x|-[R-tN]+\epsilon\right)\left(\alpha_h(t-\rho)-\alpha_h(t-\tau)\right)\Bigg[\frac{x}{|x|}\cdot q\left(x,u\left(x,t\right),\tilde{u}\left(x,t\right)\right)
+N\eta\left(u\left(x,t\right),\tilde{u}\left(x,t\right)\right)\Bigg].
\end{multline}
Recall that
\begin{equation}
q\left(x,u\left(x,t\right),\tilde{u}\left(x,t\right)\right):=\sign\left(u(x,t)-\tilde{u}(x,t)\right)\Big(f(x,u(x,t))-f(x,\tilde{u}(x,t))\Big).
\end{equation}
For almost every $(x,t)\in \mathcal{K}$, we get by the definition of the number $N$ (see \eqref{eq:definition of N in the proof of uniqueness}) 
\begin{equation}
\label{eq:bound on the entropy flux through N}
\left|\frac{x}{|x|}\cdot q\left(x,u\left(x,t\right),\tilde{u}\left(x,t\right)\right)\right|
\leq \left|f(x,u(x,t))-f(x,\tilde{u}(x,t))\right|\leq N|u(x,t)-\tilde{u}(x,t)|=N\eta\left(u(x,t),\tilde{u}(x,t)\right).
\end{equation}
Let us denote the term in the last line of \eqref{eq:the entropy integrand} by $H(x,t)$, meaning that 
\begin{equation}
H(x,t):=\omega_\epsilon\left(|x|-[R-tN]+\epsilon\right)\left(\alpha_h(t-\rho)-\alpha_h(t-\tau)\right)\Bigg[\frac{x}{|x|}\cdot q\left(x,u\left(x,t\right),\tilde{u}\left(x,t\right)\right)
+N\eta\left(u\left(x,t\right),\tilde{u}\left(x,t\right)\right)\Bigg].
\end{equation}
Note that $H$ is a non-negative function as a product of non-negative functions: the mollifier kernel $\omega_\epsilon$ is non-negative by its definition; the function $\alpha_h$ is monotonically non-decreasing and we assume that $\rho<\tau$; and the last function, $|x|^{-1}x\cdot q(x,u,\tilde{u})+N\eta(u,\tilde{u})$, is non-negative by \eqref{eq:bound on the entropy flux through N}. 

Now, substituting \eqref{eq:the entropy integrand} into \eqref{eq:subsistution into Kato inequality}, and using the non-negativity of $H$, we obtain:
\begin{equation}
\label{eq:subsistution into Kato inequality1}
\intop_0^{N^{-1}R}\intop_{B_{R-tN}(0)}\left(\omega_h(t-\rho)-\omega_h(t-\tau)\right)\chi_\epsilon(x,t)\eta\left(u\left(x,t\right),\tilde{u}\left(x,t\right)\right)dxdt\geq
\intop_0^{N^{-1}R}\intop_{B_{R-tN}(0)}H(x,t)dxdt\geq 0.
\end{equation}
Taking the limit as $\epsilon\to 0^+$ in \eqref{eq:subsistution into Kato inequality1} and taking into account \eqref{eq:limit of chi epsilon converges to the indicator of the cone} and \eqref{eq:bound on chi epsilon}, we get by Dominated Convergence Theorem
\begin{equation}
\label{eq:subsistution into Kato inequality2}
\intop_0^{N^{-1}R}\intop_{B_{R-tN}(0)}\left(\omega_h(t-\rho)-\omega_h(t-\tau)\right)\eta\left(u\left(x,t\right),\tilde{u}\left(x,t\right)\right)dxdt\geq 0.
\end{equation}
From \eqref{eq:subsistution into Kato inequality2} and linearity of integral, we get 
\begin{equation}
\label{eq:subsistution into Kato inequality3}
\intop_0^{N^{-1}R}\omega_h(t-\rho)\left[\intop_{B_{R-tN}(0)}\eta\left(u\left(x,t\right),\tilde{u}\left(x,t\right)\right)dx\right]dt\geq 
\intop_0^{N^{-1}R}\omega_h(t-\tau)\left[\intop_{B_{R-tN}(0)}\eta\left(u\left(x,t\right),\tilde{u}\left(x,t\right)\right)dx\right]dt.
\end{equation}
Recall the choice of $\tau,\rho$ as Lebesgue points (see \eqref{eq:Lebesgue points of integral bounded function} and below it). Note also that for every $h\in (0,\infty)$ with $h<\min\{\rho, N^{-1}R-\tau\}$, the intervals $(\rho-h,\rho+h)$ and $(\tau-h,\tau+h)$ are subsets of $(0,N^{-1}R)$. Taking the limit in \eqref{eq:subsistution into Kato inequality3} as $h\to 0^+$, we get
\begin{equation}
\label{eq:subsistution into Kato inequality4}
\intop_{B_{R-\rho N}(0)}\eta\left(u(x,\rho),\tilde{u}(x,\rho)\right)dx\geq 
\intop_{B_{R-\tau N}(0)}\eta\left(u(x,\tau),\tilde{u}(x,\tau)\right)dx.
\end{equation}
This completes the proof.
\end{proof}

\begin{cor}
\label{lem:$L^1$-contraction property}
(Global $L^1$-contraction property)
Let $u$,$\tilde{u}$ be two entropy solutions of $\eqref{eq: The Cauchy problem for the scalar conservation law}$, and $f$ is a flux as in Assumption \ref{ass:assumption on the flux}.  Let $M$ be as in \eqref{eq:definition of M}; for $R\in (0,\infty)$ we define $N:=N_M(R)$ as in \eqref{eq:definition of N in the proof of uniqueness}. Assume that
\begin{equation}
\label{eq:limit of NR-1}
\lim_{R\to\infty}\frac{N}{R}=0.
\end{equation}
Then, there exists a set $\mathcal{N}\subset I=(0,\infty)$ such that $\mathcal{L}^1(\mathcal{N})=0$ and 
\begin{equation}
\label{eq:global L1-contraction property}
\intop_{\R^d}\left|u(x,\tau)-\tilde{u}(x,\tau)\right|dx\leq\intop_{\R^d}\left|u(x,\rho)-\tilde{u}(x,\rho)\right|dx,
\end{equation}
for every $\rho,\tau\in I\setminus \mathcal{N}$ with $\rho\leq \tau$. 
\end{cor}

\begin{proof}
Let $\{R_j\}_{j\in\N}$ be a sequence of positive numbers which converges to infinity as $j\to \infty$. From Theorem \ref{thm:localized $L^1$-contraction property}, we have for every $j\in\N$ a set $\mathcal{N}_j\subset (0,N_j^{-1}R_j)$, $N_j:=N_M(R_j)$, such that $\mathcal{L}^1(\mathcal{N}_j)=0$ and 
\begin{equation}
\label{eq:usage of localized contraction theory}
\intop_{B_{R_j-\tau N_j}(0)}\left|u(x,\tau)-\tilde{u}(x,\tau)\right|dx\leq\intop_{B_{R_j-\rho N_j}(0)}\left|u(x,\rho)-\tilde{u}(x,\rho)\right|dx\leq \intop_{\R^d}\left|u(x,\rho)-\tilde{u}(x,\rho)\right|dx,
\end{equation}
for every $\rho,\tau\in (0,N_j^{-1}R_j)\setminus \mathcal{N}_j$ with $\rho\leq \tau$. The right-hand side of \eqref{eq:usage of localized contraction theory} can be equal to $\infty$. We get from the assumption \eqref{eq:limit of NR-1}
\begin{equation}
\label{eq:limit of NR-1 (1)}
\lim_{j\to\infty}\left(R_j-\tau N_j\right)=\lim_{j\to\infty}R_j\left(1-\tau\frac{N_j}{R_j}\right)=\infty.
\end{equation}
Let us define $\mathcal{N}:=\bigcup\limits_{j=1}^\infty\mathcal{N}_j$. Note that $\mathcal{L}^1(\mathcal{N})=0$. Let $\rho,\tau\in I\setminus \mathcal{N}$ be such that $\rho\leq \tau$. From \eqref{eq:limit of NR-1}, we get for every big enough $j$ that $0<\rho\leq \tau<N_j^{-1}R_j$. Therefore, we get from \eqref{eq:usage of localized contraction theory}
\begin{equation}
\label{eq:usage of localized contraction theory1}
\intop_{B_{R_j-\tau N_j}(0)}\left|u(x,\tau)-\tilde{u}(x,\tau)\right|dx\leq \intop_{\R^d}\left|u(x,\rho)-\tilde{u}(x,\rho)\right|dx.
\end{equation}
Taking the limit as $j\to \infty$ in \eqref{eq:usage of localized contraction theory1} and taking into account \eqref{eq:limit of NR-1 (1)}, we obtain \eqref{eq:global L1-contraction property}. Note that we do not assume global integrability of $u,\tilde{u}$.
\end{proof}

\begin{cor}[Uniqueness of entropy solutions]
Let $u$,$\tilde{u}$ be two entropy solutions of $\eqref{eq: The Cauchy problem for the scalar conservation law}$
with the same initial data $u_{0}$. Let $M$ be as in \eqref{eq:definition of M}; for $R\in (0,\infty)$ we define $N:=N_M(R)$ as in \eqref{eq:definition of N in the proof of uniqueness}. Assume that
\begin{equation}
\lim_{R\to\infty}\frac{N}{R}=0.
\end{equation}
Then, $u=\tilde{u}$ almost everywhere in $\R^d\times I$.
\end{cor}

\begin{proof}
By the contraction property, Theorem \ref{thm:localized $L^1$-contraction property}, there exists a set $\mathcal{N}\subset (0,N^{-1}R)$ such that $\mathcal{L}^1(\mathcal{N})=0$ and for $\rho,\tau\in (0,N^{-1}R)\setminus \mathcal{N}$, $\rho\leq \tau$, we have
\begin{multline}
\label{eq:subsistution into Kato inequality6}
\intop_{B_{R}(0)}|u(x,\rho)-u_0(x)|dx+\intop_{B_{R}(0)}|\tilde{u}(x,\rho)-u_0(x)|dx
\\
\geq \intop_{B_{R}(0)}|u(x,\rho)-\tilde{u}(x,\rho)|dx\geq \intop_{B_{R-\rho N}(0)}|u(x,\rho)-\tilde{u}(x,\rho)|dx\geq 
\intop_{B_{R-\tau N}(0)}|u(x,\tau)-\tilde{u}(x,\tau)|dx.
\end{multline}
Therefore, for every $\tau\in (0,N^{-1}R)\setminus\mathcal{N}$, by property \eqref{eq:continuity from the right at 0} (the assumption about the initial data $u_0$ in Definition \ref{def:(entropy-solution)}), we get from \eqref{eq:subsistution into Kato inequality6} 
\begin{equation}
\label{eq:subsistution into Kato inequality7}
\intop_{B_{R-\tau N}(0)}|u(x,\tau)-\tilde{u}(x,\tau)|dx=0.
\end{equation}
Hence  
\begin{equation}
\label{eq:subsistution into Kato inequality8}
\intop_0^{N^{-1}R}\intop_{B_{R-\tau N}(0)}|u(x,\tau)-\tilde{u}(x,\tau)|dxd\tau=0.
\end{equation}
Therefore, from \eqref{eq:subsistution into Kato inequality8} we conclude that $u=\tilde{u}$ almost everywhere in 
\begin{equation}
\mathcal{K}_R:=\Set{(x,t)\in \R^d\times I}[t\in {(0,N^{-1}R)},\,\,x\in B_{R-tN}(0)].
\end{equation}
 Let us prove that
\begin{equation}
\label{eq:representation via cones}
\R^d\times I=\bigcup_{R\in (0,\infty)}\mathcal{K}_R.
\end{equation}
Let $(x,t)\in \R^d\times I$. By the assumption that $\lim_{R\to\infty}\frac{N}{R}=0$, we get that 
\begin{equation}
\label{eq:sublinearity}
\lim_{R\to\infty}\frac{R}{N}=\infty,\quad \lim_{R\to\infty}(R-tN)=\lim_{R\to\infty}R\left(1-t\frac{N}{R}\right)=\infty.
\end{equation}
Therefore, there exists sufficiently large $R$ such that $t\in (0,N^{-1}R)$ and $|x|<R-tN$. Hence $(x,t)\in \mathcal{K}_R$. This proves \eqref{eq:representation via cones}. From \eqref{eq:sublinearity} we get the existence of an increasing to $\infty$ sequence $\{R_j\}_{j\in\N}\subset(0,\infty)$ such that the two sequences $\left\{N_j^{-1}R_j\right\}_{j\in\N}$ and $\{R_j-tN_j\}_{j\in\N}$ are monotonically increasing to $\infty$ as $j\to \infty$. Hence, for every $j<i$ we have $\mathcal{K}_{R_j}\subset \mathcal{K}_{R_i}$ and from \eqref{eq:representation via cones}, we obtain  
\begin{equation}
\label{eq:representation via countable numbers of cones}
\R^d\times I=\bigcup_{j=1}^\infty\mathcal{K}_{R_j}.
\end{equation}
Since $u=\tilde{u}$ almost everywhere in each $\mathcal{K}_{R_j}$, we get that $u=\tilde{u}$ almost everywhere in $\R^d\times I$.
\end{proof}

\textbf{Acknowledgement}
\\
This paper is based on the author's M.Sc. thesis at the Institute of Mathematics, Hebrew University of Jerusalem. It is a pleasure to thank my advisor, Professor Matania Ben-Artzi, for his continuous guidance and support.
\\
\\
\textbf{On behalf of all authors, the corresponding author states that there is no conflict of interest}

\vskip 0.3cm

Department of Mathematics, Ben-Gurion University of the Negev, P.O.Box 653, Beer Sheva 8410501, Israel\\ 
\emph{E-mail:} \textbf{pazhash@post.bgu.ac.il}


\begin{thebibliography}{100}

\bibitem{Audusse} E. Audusse and B. Perthame. Uniqueness for scalar conservation laws with discontinuous flux via adapted entropies. \textit{Proc. Roy. Soc. Edinburgh Sect. A} \textbf{135} (2005), no. 2, 253--265.
    
\bibitem{andreianov2011} B. Andreianov, K.H.Karlsen and N.H. Risebro, "A Theory of $L^1$-dissipative solvers for scalar conservation laws with discontinuous flux", Arch. Rat. Mech. Anal. {\bf 201} (2011),  27--86.
   
\bibitem{andreianov2015} B. Andreianov and D. Mitrovi\'{c}, "Entropy conditions  for scalar conservation laws with discontinuous flux revisited", Annales IHP-Anal. Nonlin. {\bf 32} (2015),1307--1335.
  
\bibitem{BachmannJulien} F. Bachmann and J. Vovelle, "Existence and uniqueness of entropy solution of scalar conservation laws with a flux function involving discontinuous coefficients", Comm. PDE {\bf 31} (2006),371--395.  
  
\bibitem{MatainaBank}
M.Bank and M.Ben-Artzi. "Scalar conservation laws on a half-line:a parabolic approach", Journal of Hyperbolic differential Equations 7.01 (2010): 165-189.  

\bibitem{BLe} M. Ben-Artzi and Philippe G. LeFloch. "Well-posedness theory for geometry-compatible hyperbolic conservation laws on manifolds." Annales de l'Institut Henri Poincar{\'e} C 24.6 (2007): 989-1008.

\bibitem{Benilan1} Ph. B{\'e}nilan, "Equations d'{\'e}volution dans un espace de Banach quelconque et applications," doctoral thesis, Universit{\'e} Paris Orsay, 1972.

\bibitem{bulicek}
BUL\'{I}\v{C}EK, Miroslav, Piotr Gwiazda, and Agnieszka \'{S}wierczewska-Gwiazda. "Multi-dimensional scalar conservation laws with fluxes discontinuous in the unknown and the spatial variable." \textit{Mathematical Models and Methods in Applied Sciences} \textbf{23}(3), 407--439 (2013).

\bibitem{Colombo} R.M. Colombo, V. Perrollaz, and A. Sylla. Conservation laws and Hamilton-Jacobi equations with space inhomogeneity. \textit{J. Evol. Equ.} \textbf{23} (2023), no. 3, Paper No. 50, 72 pp.

\bibitem{Crandall} M.G. Crandall. The semigroup approach to first order quasilinear equations in several space variables. \textit{Israel J. Math.} \textbf{12} (1972), 108--122.


\bibitem{crasta} G. Crasta, V. De Cicco and G. De Philippis, "Kinetic formulation and uniqueness for scalar conservation laws with discontinuous flux", Comm. PDE {\bf 40} (2015),694--726.

\bibitem{dafermos} C. M. Dafermos, "Hyperbolic Conservation Laws in Continuum Physics, 4-th Ed.", Grundlehren der Mathematischen Wissenschaften, vol. 325,\,Springer, 2016.
  
\bibitem{diehl} S. Diehl,  "A uniqueness condition for nonlinear convection-diffusion equations with discontinuous coefficients", J.  Hyperbolic Diff. Eqs. {\bf 06} (2009),  127--159.
   
\bibitem{Dalibard} A.-L. Dalibard. Kinetic formulation for heterogeneous scalar conservation laws. \textit{Ann. Inst. H. Poincar{\'e} C Anal. Non Lin{\'e}aire} \textbf{23} (2006), no. 4, 475--498.

   
\bibitem{Evans} L.C.Evans, "Partial Differential Equations", American Mathematical Society, 1998.
 
\bibitem{evans2015measure} L.C.Evans and R.F.Gariepy, "Measure theory and fine properties of functions"
,CRC press,2015.
 
\bibitem{F69} H.~Federer, "Geometric measure theory", Sp\-rin\-ger Verlag, Berlin, 1969. 
 
\bibitem{GagneuxMadauneTort1996}
G. Gagneux and M. Madaune-Tort,
\textit{Analyse math{\'e}matique de mod{\'e}les non lin{\'e}aires de l'ing{\'e}nierie p{\'e}troli{\'e}re}. 
(French) [Mathematical analysis of nonlinear models of petroleum engineering],
Math{\'e}matiques et Applications (Berlin), vol. 22, Springer-Verlag, Berlin, 1996.
  
\bibitem{GodlewskiRaviart} E. Godlewski and P.A.
Raviart, "Hyperbolic Systems of Conservation Laws",
Ellipses, 1991.

\bibitem{Gelfand} Gelfand, I.M. Some problems in the theory of quasi-linear equations. (Russian) \textit{Uspehi Mat. Nauk} \textbf{14} (1959), no. 2, 87--158.


\bibitem{jimenez} J. Jimenez, "Mathematical analysis of a scalar multidimensional conservation law with discontinuous flux", J. Evolution Eqs. {\bf 11} (2011), 553--576.

\bibitem{KarlsenChen2005}
Chen, Gui-Qiang and Karlsen, Kenneth. (2005). Quasilinear anisotropic degenerate parabolic equations with time-space dependent diffusion coefficients. Communications on Pure and Applied Analysis. 4. 10.3934/cpaa.2005.4.241. 

\bibitem{KarlsenRisebro2000}
Karlsen, K.~H., and N.~H. Risebro.
\newblock "On the Uniqueness and Stability of Entropy Solutions of Nonlinear
  Degenerate Parabolic Equations with Rough Coefficients."
\newblock {\em Discrete and Continuous Dynamical Systems - Series A}, 6(4),
  683--701, 2000.


\bibitem{Kruzkov}S. Kruzkov, "First-order quasilinear equations with several space variables",
Math.USSR Sb. {\bf 10} (1970), 217--243.

\bibitem{Lengeler-Muller}
D. Lengeler and T. M{\"u}ller. "Scalar conservation laws on constant and time-dependent Riemannian manifolds." Journal of Differential Equations 254.4 (2013): 1705-1727.

\bibitem{mitrovic}D. Mitrovi{\'c}, "New entropy conditions for scalar conservation laws with discontinuous flux", Disc.  Cont. Dynamical Sys. {\bf 30} (2011),  1191--1210.

\bibitem{Oleinik} O.A. Oleinik. Discontinuous solutions of non-linear differential equations. \textit{Amer. Math. Soc. Transl. Ser. 2} \textbf{26}:95--172, 1963.
   
\bibitem{otto} F. Otto,   "A regularizing effect
of nonlinear transport equations", Quart. Appl. Math. {\bf 56} (1998),
 355--375.
 
\bibitem{panov}E. Y. Panov, "On existence and uniqueness of entropy solutions to the Cauchy problem for a conservation law with discontinuous flux", J.  Hyperbolic Diff. Eqs. {\bf 06} (2009),  525--548.
   
\bibitem{Rudin}Rudin, Walter. "Real and complex analysis. 1987." Cited on 156 (1987): 16.

\bibitem{Perthame1} B. Perthame. Uniqueness and error estimates in first order quasilinear conservation laws via the kinetic entropy defect measure. \textit{J. Math. Pures Appl.} (9) \textbf{77} (1998), no. 10, 1055--1064.

\bibitem{Perthame2} B. Perthame. Kinetic formulation of conservation laws. \textit{Oxford Lecture Ser. Math. Appl.}, \textbf{21} Oxford University Press, Oxford, 2002.


\bibitem{SeguinAnslysis} N. Seguin and J. Vovelle, "Analysis and approximation of scalar conservation law with a flux function with discontinuous coefficients", Math. Models and Methods  Appl. Sci.  {\bf 13} (2003),221--257.
 
\bibitem{shen2012}
Shen, Chun, and Meina Sun. "The bifurcation phenomenon for scalar conservation laws with discontinuous flux functions." \textit{Acta Applicandae Mathematicae} \textbf{121} (2012): 69--80.
 
\bibitem{shen2013} C. Shen and M. Sun, "Wave interactions and stability of the Riemann solutions for a scalar conservation law with a discontinuous flux function", Zeits.  Angew. Mathematik  Physik {\bf 64} (2013), 1025--1056.

\bibitem{vol'pert} A.I. Vol'pert,
"The spaces BV and quasilinear equations", Math. USSR Sb. {\bf 2} (1967), 225-
267
\end{thebibliography}
\end{document}